\documentclass{article}

\usepackage{amsmath}
\usepackage{amsthm}
\usepackage{amsfonts}
\usepackage{amsbsy}
\usepackage{amssymb}
\usepackage[initials]{amsrefs}
\newtheorem{theorem}{Theorem}
\newtheorem{proposition}{Proposition}

\newtheorem{lemma}{Lemma}

\newcommand{\R}{{\mathbb R}}
\newcommand{\Z}{{\mathbb Z}}
\newcommand{\set}[2]{ \left\{ #1 \ \left| \ #2 \right. \right\} }

\usepackage{tikz}
\usetikzlibrary{decorations.pathreplacing}

\newtheorem*{theorem*}{Theorem}

\newcommand{\GL}[1]{\mathrm{GL}_{#1}}
\newcommand{\ang}[1]{\left< #1 \right>}
\newcommand{\tr}{\mathrm{tr}}

\newcommand{\SL}[1]{\mathrm{SL}_{#1}}
\newcommand{\li}[1] {{{}^{#1}\Sigma}}
\newcommand{\ri}[1] {\Sigma^{#1}} 
\newcommand{\inc}{\Sigma}
\newcommand{\blw}{\mathrm{W}}

\newcommand{\ms}{\cdot {\boldsymbol \partial}}
\newcommand{\mso}{{\boldsymbol \partial}}
\newcommand{\MS}{{\mathbb M}}

\title{$L^p$-improving estimates for Radon-like operators and the Kakeya-Brascamp-Lieb inequality}
\author{Philip T. Gressman}
\date{\today}

\begin{document}
\maketitle
\begin{abstract}
This paper considers the problem of establishing $L^p$-improving inequalities for Radon-like operators in intermediate dimensions (i.e., for averages overs submanifolds which are neither curves nor hypersurfaces). Due to limitations in existing approaches, previous results in this regime are comparatively sparse and tend to require special numerical relationships between the dimension $n$ of the ambient space and the dimension $k$ of the submanifolds. This paper develops a new approach to this problem based on a continuum version of the Kakeya-Brascamp-Lieb inequality, established by Zhang \cite{zhang2018} and extended by Zorin-Kranich \cite{zk2018}, and on recent results for geometric nonconcentration inequalities \cite{gressman2018}. As an initial application of this new approach, this paper establishes sharp restricted strong type $L^p$-improving inequalities for certain model quadratic submanifolds in the range $k < n \leq 2k$.
\end{abstract}

\tableofcontents

\section{Introduction}
\subsection{Background and statement of results}

$L^p$-improving estimates for Radon-like operators have been a fundamental object of study in harmonic analysis for many decades and find applications in a number of interesting problems in PDEs and elsewhere (see, e.g., \cite{patel2018}). Since the late 1990s, a favored approach has been a combinatorial one, pioneered by Christ \cite{christ1998}, who was inspired by Bourgain \cite{bourgain1991,bourgain1986}, Wolff \cite{wolff1995,wolff1997}, and Schlag \cite{schlag1997}, as well as others.  As this approach is commonly executed, it involves the construction of a so-called ``inflation map'' which iterates the geometry of the operator in much the same way that a $TT^*$ argument would. A key feature of the inflation map is that the dimension of its domain (usually comprised of products of fibers) and its target space must generally match and, when they do, the map must have a Jacobian determinant which is nonzero on a dense open set. The difficulty of completing a proof, once the inflation map has been obtained, boils down to a delicate understanding of how the degeneracy of the Jacobian determinant leads to certain integral inequalities.

A principal limitation of this approach is that inflation maps are often difficult to construct or analyze unless the dimension and the codimension of the underlying submanifolds happen to satisfy simple numerical relationships, e.g., when one is an integer multiple of the other. For this reason, there are many gaps in the literature for Radon-like operators of intermediate dimension (being neither curves nor hypersurfaces) when the dimension and codimension are generically chosen.

In this paper, we introduce a new approach to this problem which allows one to circumvent the need for an explicit inflation map. The overall philosophy of the proof is still fundamentally combinatorial and very deeply connected to earlier approaches, but incorporates recent ideas including the so-called Kakeya-Brascamp-Lieb inequality, proved by Zhang \cite{zhang2018} and further developed by Zorin-Kranich \cite{zk2018}, and nonconcentration inequalities \cite{gressman2018}. The result is a significant shift in the structure of the argument which removes a number of important barriers and gives a unified framework which applies across a number of situations with wildly different inflation maps (or no known inflation map at all).

Central to this approach is a new understanding of the Brascamp-Lieb constant. To define it in a form which is most suitable for the present purposes, let $m$, $n$, and $k$ be positive integers with $n > k$ and suppose that $\pi_1,\ldots,\pi_m$ are linear maps from $\R^n$ to $\R^{n-k}$. Fix $p := \frac{n}{m(n-k)}$. Let $\blw(\{\pi_j\}_{j=1}^m)$, which will be called the Brascamp-Lieb weight associated to the maps $\{\pi_j\}_{j=1}^m$, be defined to equal the largest nonnegative real number such that
\begin{equation} \blw(\{\pi_j\}_{j=1}^m) \int_{\R^n} \left[ \prod_{j=1}^m f_j( \pi_j x ) \right]^{p} dx  \leq \left[ \prod_{j=1}^m \int_{\R^{n-k}}  f_j \right]^p \label{datadef} \end{equation}
holds for all nonnegative measurable functions $f_j$ on $\R^{n-k}$, $j=1,\ldots,m$.

At the greatest level of generality, the results of this paper are simplest to state for Radon-like operators which are defined in terms of an incidence relation $\Sigma$ which is itself understood to be the zero set of a defining function $\rho$. More precisely,
let $\Omega \subset \R^n \times \R^n$, and let  $\rho : \Omega \rightarrow \R^{n-k}$ be a smooth function such that at every point $(x,y) \in \Omega$ such that $\rho(x,y) = (\rho_1(x,y),\ldots,\rho_{n-k}(x,y)) = 0$, the  matrices
\begin{equation}  D_x \rho := \left[ \begin{array}{ccc} 
\frac{\partial \rho_{1}}{\partial x_1} & \cdots & \frac{\partial \rho_{1}}{\partial x_n} \\
\vdots & \ddots & \vdots \\
\frac{\partial \rho_{n-k}}{\partial x_1} & \cdots & \frac{\partial \rho_{n-k}}{\partial x_n} 
\end{array} \right]  \mbox{ and } D_y \rho :=  \left[ \begin{array}{ccc} 
\frac{\partial \rho_{1}}{\partial y_1} & \cdots & \frac{\partial \rho_{1}}{\partial y_n} \\
\vdots & \ddots & \vdots \\
\frac{\partial \rho_{n-k}}{\partial y_1} & \cdots & \frac{\partial \rho_{n-k}}{\partial y_n} 
\end{array} \right] \label{incmat} \end{equation}
(which will be called the left and right derivative matrices of $\rho$, respectively) both have full rank $n-k$. We call the set $\inc := \set{ (x,y) \in \Omega}{ \rho(x,y) = 0}$ the incidence relation associated to $\rho$ and call $\rho$ a defining function of the incidence relation $\inc \subset \Omega$. By virtue of the Implicit Function Theorem, the sets 
\[ \li{x} := \set{ y \in \R^n}{ (x,y) \in \Omega \mbox{ and } \rho(x,y) = 0} \]
and
\[ \ri{y} := \set{ x \in \R^n}{ (x,y) \in \Omega \mbox{ and } \rho(x,y) = 0} \]
are embedded $k$-dimensional submanifolds of $\R^n$ for any values of the parameters $x$ or $y$, respectively.  The incidence relation $\inc$ will be called left-algebraic of degree $d$ when for each $y$ such that $\ri{y}$ is nonempty, $\ri{y}$ is contained in a $k$-dimensional affine algebraic variety of degree at most $d$ (where we do not distinguish between affine algebraic sets and affine algebraic varieties and do not require irreducibility).  It is also important to define a canonical measure $d \sigma$ on each $\li{x}$ by means of the formula
\begin{equation} \int_{\li{x}} f d \sigma := \int_{\li{x}} f(y) \frac{d \mathcal H^k(y)}{\det (D_y \rho(x,y) (D_y \rho (x,y))^T)^{1/2}}, \label{themeasure} \end{equation}
where $d \mathcal H^k$ is the usual $k$-dimensional Hausdorff measure restricted to $\li{x}$. Analogous measures on $\ri{y}$ may be defined as well, but will not be needed. 

The first main result of this paper is the following continuum version of the Kakeya-Brascamp-Lieb inequality:
\begin{theorem}
Suppose $\inc$ is a left-algebraic incidence relation of degree $d$ with defining function $\rho$. 
Then for any nonnegative Lebesgue integrable functions $f_1,\ldots,f_m$ on $\R^n$, \label{kakeyathm}
\begin{equation} 
\begin{aligned} \int_{\R^n} \left[ \int_{\li{x}} \! \! \!  \! \cdots \int_{\li{x}} \left[\blw(\{D_x \rho(x,y_j)\}_{j=1}^m) \right]^{\frac{1}{p}}  \prod_{j=1}^m f_j(y_j) ~ d \sigma(y_1) \cdots d \sigma(y_m) \right]^{p} \! \!  dx & \\ \leq  C \prod_{j=1}^m  \left( \int f_j  \right)^{p} & \end{aligned} \label{kakeya} \end{equation}
for some $C < \infty$ depending only on $n$, $m$, and $d$, where $\blw(\{D_x \rho(x,y_j)\}_{j=1}^m)$ is the constant as defined by \eqref{datadef} when $\pi_j := D_x \rho(x,y_j)$ for each $j=1,\ldots,m$.
\end{theorem}

The inequality \eqref{kakeya} is the main new tool of this paper for studying the $L^p$-improving properties of Radon-like operators in intermediate dimensions. When combined with recent new machinery regarding nonconcentration functionals \cite{gressman2018},  the inequality \eqref{kakeya} can be used as a direct replacement for an inflation map construction and the associated degenerate change of variables formula. This overcomes some significant limitations of that approach in the regime of intermediate dimensions. The most general result of this paper concerning $L^p$-improving properties is the following:
\begin{theorem}
Suppose $\Sigma \subset \Omega$ is a left-algebraic incidence relation with defining function $\rho$. Suppose also that \label{genradonthm}
\begin{equation} \sup_{y_1,\ldots,y_m \in F} [ W ( \{D_x \rho(x,y_j)\}_{j=1}^m) ]^{\frac{1}{p}} \gtrsim (\sigma ( F \cap \li{x}))^{s} \label{nonconcentrate} \end{equation}
for all $x \in \R^n$ and all Borel subsets $F \subset \li{x}$, where $\sigma$ is the measure \eqref{themeasure}. Then the Radon-like transform
\begin{equation} T f(x) := \int_{\li{x}} f d \sigma \label{theradondef} \end{equation}
satisfies the inequality
\begin{equation} \left( \int | T \chi_E(x)|^{\frac{n(m+s)}{(n-k)m}} dx \right)^{\frac{m(n-k)}{n(m+s)}} \leq C |E|^{\frac{m}{m+s}} \label{theradonineq} \end{equation}
for all Borel sets $E \subset \R^n$ with constant $C$ which depends only on $n, k, m,s,$ and the degree of $\Sigma$. Here $|E|$ denotes the Lebesgue measure of $E$.
\end{theorem}
We call \eqref{theradonineq} a restricted strong type $(\frac{m+s}{m}, \frac{n(m+s)}{(n-k)m})$ inequality for $T$ following usual conventions, e.g. \cite{bos2009}.  In Section \ref{radonsec} we give several examples of how one can verify the main hypothesis \eqref{nonconcentrate} in a number of important special cases. The broadest of these applications is:
\begin{theorem}
For any integers $n,k$ satisfying $k < n \leq 2k$, consider the Radon-like operator acting on functions on $\R^n$ given by \label{radonthm}
\begin{equation}
\begin{split} T f(x) := \int_{\R^k} f & \left(x_1 + t_1,  \ldots, x_k + t_k, \vphantom{x_{k+1} + \frac{1}{2} \sum_{i=1}^k \lambda_{1i} t_i^2,\ldots, x_{n} + \frac{1}{2} \sum_{i=1}^k \lambda_{(n-k)i} t_i^2} \right. \\
& \qquad \left. x_{k+1} + \frac{1}{2} \sum_{i=1}^k \lambda_{1i} t_i^2,\ldots, x_{n} + \frac{1}{2} \sum_{i=1}^k \lambda_{(n-k)i} t_i^2  \right) dt,
\end{split} \label{myoperator}
\end{equation}
where $\lambda_{ji}$ is a $(n-k) \times k$ matrix whose minors satisfy the constraint
\[ \det \left[ \begin{array}{ccc} \lambda_{1 i} & \cdots & \lambda_{1 (i+n-k-1)} \\ \vdots & \ddots & \vdots \\
\lambda_{(n-k) i} & \cdots & \lambda_{(n-k) (i+n-k-1)} \end{array} \right] \neq 0 \]
for all $i$ (interpreting the columns as periodic with period $k$ to make sense of the index $i + n - k -1$ when $i + n-k - 1 > k$).  Then for all Borel sets $E \subset \R^n$, 
\begin{equation} || T \chi_E ||_{L^{\frac{2n-k}{n-k}}(\R^n)} \leq C |E|^{\frac{n}{2n-k}} \label{knapp} \end{equation}
for some $C < \infty$ independent of $E$.
\end{theorem}
A standard Knapp-type argument shows that the exponents in the conclusion \eqref{knapp} cannot be improved; as such, Theorem \ref{radonthm} can be regarded as an extension of work of by D. Oberlin \cite{oberlin2008} concerning ``model surface'' quadratic submanifolds. We note that it is understood through work of Ricci \cite{ricci1997} that quadratic model surfaces exist with dimension $k$ much less than $n/2$ when $n$ is large; the restriction $n \leq k/2$ present in Theorem \ref{radonthm} is not a fundamental limitation of the method; in particular,  Section \ref{radonmaxsec} illustrates how the method can be applied to a canonical non-translation-invariant quadratic Radon-like operator which integrates over submanifolds of dimension $k$ and codimension $k^2$.

\subsection{Outline and notation}
The remainder of this paper is organized as follows: Section \ref{kakeyasec} contains the proof of Theorem \ref{kakeyathm}, which is derived from a discrete inequality of Zhang and Zorin-Kranich using a host of essentially standard limiting arguments. Section \ref{blsec} proves a number of important new results about the nature of the Brascamp-Lieb constant. In the context of Theorem \ref{genradonthm}, the most important of these is Lemma \ref{bl2polylem}, which establishes the comparability of the Brascamp-Lieb constant and a supremum of certain invariant polynomials. The approach is to observe a deep connection between the Brascamp-Lieb constant and the field of Geometric Invariant Theory.  Lemma \ref{itsadet} also gives important insight into the family of these invariant polynomials, and in particular establishes that each such polynomial can be expressed as the determinant of a matrix with certain simple block structure, which is particularly useful when seeking to apply Theorem \ref{genradonthm}. Section \ref{proofsec} gives the proof of Theorem \ref{genradonthm}. The proof is a relatively straightforward combination of Theorem \ref{kakeyathm}, Lemma \ref{bl2polylem} and Proposition \ref{convprop}, which is itself a generalization of a result which was central to the study of nonconcentration inequalities \cite{gressman2018}. Section \ref{appsec} provides a number of sample applications of Theorem \ref{genradonthm} which include the moment curve case studied by Christ \cite{christ1998}, Theorem \ref{radonthm}, and some non-translation-invariant extensions. Finally, Section \ref{appendix} is an appendix which provides some elementary quantitative versions of the Inverse and Implicit Function Theorems which are needed in the proof of Theorem \ref{kakeyathm}.

The remainder of this paper employs the notation $\lesssim$ as is now rather commonly done: the statement $A \lesssim B$ will mean that there exists a finite nonnegative constant $C$ such that $A \leq C B$ holds uniformly over some range of parameters of $A$ and $B$. When those parameters are not readily apparent, they will be explicitly identified, e.g., ``$A_j \lesssim B_j$ uniformly for all $j$.'' The notation $A \gtrsim B$ is defined analogously, and $A \approx B$ will be used to indicate that both $A \lesssim B$ and $A \gtrsim B$ hold simultaneously.

Another important piece of space-saving notation which will be used heavily is the following: for any objects $p_1,\ldots,p_m$, the notation $\{p_j\}_{j=1}^m$ will denote the $m$-tuple $(p_1,\ldots,p_m)$.

\section{Continuous Kakeya-Brascamp-Lieb: Proof of Theorem \ref{kakeyathm}}
\label{kakeyasec}
The core result of this section is the proof of Theorem \ref{kakeyathm}. Our derivation is based directly on the Kakeya-Brascamp-Lieb inequality of Zorin-Kranich \cite{zk2018}, which is a natural evolution of an earlier result of Zhang \cite{zhang2018}. Zhang's result was itself inspired by Guth's approach to endpoint multilinear Kakeya \cite{guth2010}, which was prompted by and built upon work of Bennett, Carbery, and Tao in the non-endpoint case \cite{bct2006}.

\subsection{Reduction to smooth functions}

The first step in the proof of Theorem \ref{kakeyathm} is to show that it suffices to prove \eqref{kakeya} for nonnegative smooth functions $f_j$ of compact support. This follows by standard arguments, but as $p$ will generally be less than one, it is reasonable to proceed carefully nevertheless. The auxiliary result needed is that for any nonnegative Lebesgue integrable function $f$ on $\R^n$ and any $\delta > 0$, there is a pointwise nondecreasing sequence $f_\ell$ of nonnegative smooth functions of compact support such that
\[ f(x) \leq \lim_{\ell \rightarrow \infty} f_\ell(x) \mbox{ for all } x \in \R^n\]
(as opposed to merely almost everywhere) such that 
\[ \int f_\ell \leq \delta + \int f \]
for all $\ell$. To establish this auxiliary result, let $\eta > 0$ be a positive real number satisfying 
\[ (1 + \eta) \int f \leq \frac{\delta}{3} +  \int f  \]
and let $F_j := \set{x \in \R^n}{ (1+\eta)^{j-1} < f(x) \leq ( 1 + \eta)^j}$. By definition of these sets, one has the trivial inequality
\[ f(x) \leq \sum_{j = -\infty}^{\infty} (1 + \eta)^j \chi_{F_j} (x) \]
for every $x \in \R^n$ (where the sum is interpreted as an extended real number).
Next, for each $j \in \Z$, let $O_j$ be an open set containing $F_j$, each chosen so that
\[ \sum_{j} (1 + \eta)^{j}  \left| O_j \setminus F_j \right| \leq \frac{\delta}{3}. \]
Decompose each $O_j$ into nonoverlapping dyadic boxes $Q_{jk}$ (i.e., boxes of the form $[k_1 2^{\ell},(k_1+1) 2^\ell] \times \cdots \times [k_n 2^{\ell}, (k_n+1) 2^{\ell}]$ for integers $k_1,\ldots,k_n$ and $\ell$), and for each dyadic box, select a smooth nonnegative function of compact support $\varphi_{jk}$ which is identically $1$ on $Q_{jk}$ in such a way that the entire ensemble of functions satisfies
\[ \sum_{j,k} (1 + \eta)^j \int_{\R^n \setminus Q_{jk}} \varphi_{jk} \leq \frac{\delta}{3}. \]
To bound $f$ everywhere by the limit of an appropriate nondecreasing sequence $f_\ell$, one may simply select some ordering of the countably many dyadic boxes $Q_{jk}$ and let $f_n$ be the sequence of partial sums of $(1 + \eta)^j \varphi_{jk}$.  The conclusion that $\lim_{\ell \rightarrow \infty} f_\ell(x)$ is greater than $f(x)$ at every point follows directly from the fact that $\varphi_{jk} \geq 1$ on $Q_{jk}$ and the union of the $Q_{jk}$'s contains $F_j$ for each $j$. Similarly,
\begin{align*}
 \lim_{\ell \rightarrow \infty} \int f_\ell = \int \sum_{j,k} (1 + \eta)^j \varphi_{jk} & \leq \sum_{j,k} (1 + \eta)^j \left[ |Q_{jk}| + \int_{\R^n \setminus Q_{jk}} \varphi_{jk} \right] \\
 & \leq \frac{\delta}{3} + \sum_{j} (1 + \eta)^j |O_j|   \\
 & \leq \frac{\delta}{3} + \sum_{j} (1 + \eta)^j \left[ |F_j| + |O_j \setminus F_j|  \right] \\
 & \leq \frac{2 \delta}{3}  + (1 + \eta) \sum_{j} (1 + \eta)^{j-1} |F_j| \\
 & \leq \frac{2 \delta}{3} + (1 + \eta) \int f  \leq \delta + \int f.
 \end{align*}
Assuming that \eqref{kakeya} holds for all $m$-tuples of smooth nonnegative functions of compact support, the passage to general integrable functions is achieved by an application of the Monotone Convergence Theorem (which applies because $p > 0$) for the particular approximating sequences just constructed, one for each of the $m$ functions appearing in \eqref{kakeya}, and then letting $\delta \rightarrow 0^+$.

\subsection{Kakeya-Brascamp-Lieb for functions of varieties}

After restricting attention to smooth functions of compact support, the next significant step in the proof of \eqref{kakeya} builds on the following special case of the Kakeya-Brascamp-Lieb inequality as established by Zorin-Kranich \cite{zk2018}, which is itself a generalization of the closely related Theorem 8.1 of Zhang \cite{zhang2018}:
\begin{theorem*}[Theorem 1.7 of \cite{zk2018}]
Let $\mathcal Q$ be the collection of all boxes $[j_1,j_1+1] \times \cdots \times [j_n, j_n + 1]$ for integers $j_1,\ldots,j_n$ and suppose that $H_1,\ldots,H_m$ are affine algebraic varieties in $\R^n$ with $\dim H_j = k$. Then
\begin{equation}
\begin{aligned} \sum_{Q \in \mathcal Q} \left( \int_{\prod_{j=1}^m  (H_j \cap Q)} \left[ \blw( \{ T_{x_j} H_j \}_{j=1}^m) \right]^{\frac{1}{p}} d \mathcal H^{k}(x_1) \cdots d \mathcal H^{k}(x_m) \right)^{p} & \\
\leq C_n \prod_{j=1}^m & (\deg H_j)^{p}. 
\end{aligned} \label{zkineq} \end{equation}
Here $p$ and $\blw$ are as in \eqref{datadef}, and for each smooth point $x_j$ of $H_j$, $T_{x_j} H_j$ denotes the orthogonal projection from $\R^n$ onto the orthogonal complement of the tangent space of $H_j$ at $x_j$. The constant $C_n$ depends only on $n$.
\end{theorem*}

The proof of Theorem \ref{kakeyathm} proceeds by deducing some self-improvements of the above theorem which generalize it first to a discrete weighted version of Theorem \ref{kakeyathm} and then to a continuous analogue. These refinements are the contents of the upcoming Propositions \ref{zkprop1} and \ref{zkprop2}, respectively.

For convenience in the arguments that follow, let $Q_0$ be the box $[-1/2,1/2]^n$ and let $Q_x = x + Q_0$ for all $x \in \R^n$. The norm $| \cdot |$ on $\R^n$ will denote the $\ell^\infty$ norm in the standard coordinate basis. Furthermore, given $x \in \R^n$ and an $m$-tuple $\{H_j\}_{j=1}^m$ of affine algebraic varieties in $\R^n$, define
\begin{equation}
 \omega_x(\{H_j\}_{j=1}^m) :=  \int_{\prod_{j=1}^m (H_j \cap Q_x)} \! \! \left[ \blw( \{ T_{x_j} H_j \}_{j=1}^m) \right]^{\frac{1}{p}} d \mathcal H^{k}(x_1) \cdots d \mathcal H^{k}(x_m). \label{newwt}
 \end{equation}

\begin{proposition}
If $E_1,\ldots,E_m$ are finite sets of $k$-dimensional varieties in $\R^n$ and if $N_j : E_j \rightarrow \R_{\geq 0}$ for each $j=1,\ldots,m$, then \label{zkprop1}
\begin{equation} \begin{aligned} \int_{\R^n} \left[ \mathop{\sum_{H_1 \in E_1,\ldots,}}_{H_m \in E_m} \left( \prod_{j=1}^m N_j(H_j) \right)  \omega_x(\{H_j\}_{j=1}^m) \right]^p  & dx  \\ \leq C_n \prod_{j=1}^m  & \left[ \sum_{H \in E_j} N_j(H) \deg H \right]^p \! \!, \end{aligned} \label{zkineq2}
\end{equation}
where $p$ and $C_n$ are the same as in \eqref{zkineq}.
\end{proposition}
\begin{proof}
The first step of this proposition is to replace the sum over $Q \in \mathcal Q$ in \eqref{zkineq} by an integral as in \cite{zhang2018}.
To do this, let $x \in \R^n$ be fixed and apply \eqref{zkineq} to the shifted varieties $\{-x + H_j\}_{j=1}^m$; note that shifting does not change degree. For any $Q \in \mathcal Q$,
\begin{align*}
& \int_{(-x + H_1) \cap Q} \! \! \! \cdots \int_{(-x + H_m) \cap Q} \! \! \left[ \blw( \{ T_{x_j} (-x+H_j) \}_{j=1}^m) \right]^{\frac{1}{p}} d \mathcal H^{k}(x_m) \cdots d \mathcal H^{k}(x_1) \\
& =  \int_{H_1 \cap (x+ Q)} \! \! \! \cdots \int_{H_m \cap (x+Q)} \! \! \left[ \blw( \{ T_{x_j} H_j \}_{j=1}^m) \right]^{\frac{1}{p}} d \mathcal H^{k}(x_m) \cdots d \mathcal H^{k}(x_1)
\end{align*}
by translation-invariance of Hausdorff measure.
Since the sum of this quantity over $Q \in \mathcal Q$ is bounded by $C_n \prod_{j} (\deg H_j)^p$ for all $x \in \R^n$, it follows that
\[ \sum_{\ell \in \Z^n} \left[  \omega_{x+\ell}(\{H_j\}_{j=1}^m) \right]^p \leq C_n \prod_{j=1}^m (\deg H_j)^{p} \]
for all $x \in \R^n$. Integrating $x$ over $[0,1]^n$ gives
\begin{align}
\int_{\R^n} \left[  \omega_x(\{H_j\}_{j=1}^m) \right]^p   dx  \leq C_n \prod_{j=1}^m (\deg H_j)^{p} \label{zkineq15}
\end{align}
for any $m$-tuple of affine varieties $H_1,\ldots,H_m$.

The next step is to introduce the weights $N_j$. To that end, suppose initially that $N_j$ is any nonnegative integer-valued function on $E_j$ for each $j=1,\ldots,m$. For any fixed $\delta \in (0,1)$ and each $j=1,\ldots,m$, let $\tilde H_j$ be a union of varieties of the form $u_{ji} + (1-\delta) H_j$ as $H_j$ ranges over all varieties in $E_j$ with $N_j(H_j) > 0$ and as $i$ ranges over $\{1,\ldots,N_j(H_j)\}$. Assume also that the shifts $u_{ji}$ satisfy $|u_{ji}| < \delta/2$ and are chosen so that no two of the varieties $u_{ji} + (1-\delta) H_j$ are equal. The key idea in the proof of this proposition is to apply \eqref{zkineq15} to the varieties $\tilde H_j$. First observe that $\omega_x( \{ \tilde H_j \}_{j=1}^m)$ expands as a sum of terms of the form $\omega_x( \{ u_{ji} + (1 -\delta) H_j \}_{j=1}^m)$, where for each $j$, $u_{ji} + (1-\delta) H_j$ is one of the varieties just described whose union is $\tilde H_j$. Each such term $\omega_x( \{ u_{ji} + (1 -\delta) H_j \}_{j=1}^m)$ is itself an integral over $((u_{1i} + (1-\delta) H_1) \cap Q_x) \times \cdots ((u_{mi} + (1-\delta) H_m) \cap Q_x)$ of the corresponding weight $\blw^{1/p}$ generated by the orthogonal projections onto the orthogonal complement of the tangent spaces $T_{x_j} (u_{ji} + (1-\delta) H_j)$.
Observe that $(u_{ji} + (1-\delta) H_j) \cap Q_x = u_{ji} + ((1-\delta) H_j) \cap Q_{x - u_{ji}} = u_{ji} + (1-\delta) ( H_j \cap (1-\delta)^{-1} Q_{x-u_{ji}})$ and that $(1-\delta)^{-1} Q_{x-u_{ji}} \supset Q_{(1-\delta)^{-1} x}$. To see this last fact, note that
\[ (1 - \delta)^{-1} x + y = (1-\delta)^{-1} (x - u_{ji}) + (1-\delta)^{-1} u_{ji} + y, \]
and when $|y| < 1/2$, it must follow that $|(1-\delta)^{-1} u_j + y| \leq \delta (1-\delta)^{-1}/2 + 1/2 = (1-\delta)^{-1}/2$, so that 
\[ (1-\delta)^{-1} x + y = (1-\delta)^{-1} \left( x - u_{ji} + \tilde y \right) \]
for some $|\tilde y| \leq 1/2$.
These elementary observations combined with a sequence of changes of variables imply that
\begin{align*}
& \int_{\prod_{j=1}^m (( u_{ji} + H_j) \cap Q_x)} \left[ \blw(\{ T_{x_j} (u_{ji} + (1-\delta) H_j)\}_{j=1}^m) \right]^{\frac{1}{p}} d \mathcal H^{k}(x_1) \cdots d \mathcal H^{k}(x_m) \\
& \geq \int_{\prod_{j=1}^m (1-\delta)(H_j \cap  Q_{(1-\delta)^{-1} x}) } \left[ \blw(\{ T_{x_j} ((1-\delta) H_j)\}_{j=1}^m) \right]^{\frac{1}{p}} d \mathcal H^{k}(x_1) \cdots d \mathcal H^{k}(x_m) \\
& = (1-\delta)^{mk} \int_{\prod_{j=1}^m (H_j \cap  Q_{(1-\delta)^{-1} x}) } \left[ \blw(\{ T_{x_j}  H_j\}_{j=1}^m) \right]^{\frac{1}{p}} d \mathcal H^{k}(x_1) \cdots d \mathcal H^{k}(x_m),
\end{align*}
i.e.,
\[ \omega_{(1-\delta)^{-1} x} ( \{H_j\}_{j=1}^m) \leq (1-\delta)^{-km} \omega_x( \{ u_{ji} + (1 -\delta) H_j \}_{j=1}^m). \]
Summing over the varieties forming each $\tilde H_j$ gives
\begin{equation} \sum_{H_1 \in E_1,\ldots,H_m \in E_m} N_j(H_j) \omega_{(1-\delta)^{-1} x} ( \{H_j\}_{j=1}^m) \leq (1-\delta)^{-km} \omega_x( \{ \tilde H_j \}_{j=1}^m). \label{zkscale1} \end{equation}
Since $\deg \tilde H_j \leq \sum_{H_j \in E_j} N_j(H_j) \deg H_j$, applying \eqref{zkineq15} to the varieties $\tilde H_j$, invoking the inequality \eqref{zkscale1}, applying a change of variables in $x$, and sending the spacing parameter $\delta \rightarrow 0^+$ gives the conclusion of this proposition when $N_j$ is integer-valued.

Because both sides of this inequality are homogeneous of degree $p$ with respect to each $N_j$, multiplying each $N_j$ by a nonzero real number preserves both sides of the inequality, meaning the inequality remains true when each $N_j$ is a positive real multiple of a nonnegative integer-valued function. However, every nonnegative real-valued function $N_j$ is uniformly comparable to such a function with constants which are as close as desired to $1$. Therefore the proposition must be true in the general case of each $N_j$ being an arbitrary nonnegative real-valued function.
\end{proof}

\begin{proposition}
For each $j=1,\ldots,m$, let $U_j \subset \R^n$ be an open set and let $H_j$ be a mapping from $U_j$ into the set of $k$-dimensional varieties on $\R^n$ of degree at most $D_j$ such that $H_j(y)$ depends smoothly on $y$. For any nonnegative measurable functions $f_j$ on $U_j$,  \label{zkprop2}
\begin{equation} \begin{aligned} \int_{\R^n} \left[ \int  \left( \prod_{j=1}^m f_j(y_j) \right) \omega_x( \{H_j(y_j)\}_{j=1}^m) dy_1 \cdots d y_m \right]^p  dx & \\  \leq C_n \prod_{j=1}^m  \left[ D_j \int f_j \right]^p & , \end{aligned} \label{zkineq25}
\end{equation}
where $p$ and $\omega_x$ are as above. The constant $C_n$ is the same as in Proposition \ref{zkprop1}.
\end{proposition}
\begin{proof}
Because $H_j(y_j)$ depends smoothly on $y$,  $\omega_x(\{H_j(y_j)\}_{j=1}^m$ is known to be a continuous function of $y_1,\ldots,y_m$ as a result work by Bennett, Bez, Cowling, and Flock \cite{bbcf2017}.
For any $\delta > 0$, decompose $\R^n$ into a nonoverlapping union of boxes of side length $\delta$. Fix arbitrary compact sets $K_j \subset U_j$ and let ${\mathcal Q}_j(\delta)$ be a finite collection of these cubes which covers $K_j$. For each $j$, let $E_j$ be the collection of varieties given by
\[ E_j := \set{ H }{ H = H_j(y) \mbox{ for } y \mbox{ at the center of a cube } Q' \in {\mathcal Q}_j(\delta)}. \]
For convenience, let $H_j(Q')$ also denote the variety $H_j(y)$ when $y$ is taken to be the center of $Q'$.
Fix any nonnegative measurable functions $f_j$ on $U_j$ and let 
\[ N_j(H) := \mathop{\sum_{Q' \in {\mathcal Q}_j(\delta)}}_{H_j(Q') = H} \int_{Q' \cap K_j} f_j \]
The left-hand side of \eqref{zkineq2} is exactly equal to 
\begin{equation} \int \left[ \int_{K_m} \cdots \int_{K_1} \prod_{j=1}^m f_j(y_j) \omega_x(\{H_j(y_j')\}_{j=1}^m) d y_1 \cdots dy_m \right]^p dx, \label{zkineq3} \end{equation}
where $y_j'$ is the center of the cube $Q' \in Q_j(\delta)$ containing $y_j$ (which is uniquely defined for a.e. $y_j$). By Monotone Convergence and continuity of the reciprocal of the Brascamp-Lieb constant, 
\[
\begin{aligned}
\int & \left[ \int_{K_m} \cdots \int_{K_1} \prod_{j=1}^m f_j(y_j) \omega_x(\{H_j(y_j)\}_{j=1}^m) d y_1 \cdots dy_m \right]^p dx \\
= & \int  \lim_{\delta \rightarrow 0^+} \left[ \int_{K_m} \cdots \int_{K_1} \prod_{j=1}^m f_j(y_j) \inf_{|z_j - y_j| \leq \delta} \omega_x(\{H_j(z_j)\}_{j=1}^m) d y_1 \cdots dy_m \right]^p dx \\
= & \lim_{\delta \rightarrow 0^+} \int \left[ \int_{K_m} \cdots \int_{K_1} \prod_{j=1}^m f_j(y_j) \inf_{|z_j - y_j| \leq \delta} \omega_x(\{H_j(z_j)\}_{j=1}^m) d y_1 \cdots dy_m \right]^p dx.
\end{aligned}
\]
For each $\delta > 0$, 
 \[ \inf_{|z_j - y_j| \leq \delta} \omega_x(\{H_j(z_j)\}_{j=1}^m) \leq \omega_x(\{H_j(y_j')\}_{j=1}^m) \]  %\eqref{zkineq3}
 because $|y -y'| \leq \delta$. But then by \eqref{zkineq2}, this means that the limit of \eqref{zkineq3} as $\delta \rightarrow 0^+$ is dominated by
\[  C_n \prod_{j=1}^m \left[ D_j \sum_{Q' \in \mathcal Q_j(\delta)} \int_{Q' \cap K_j} f_j \right]^p = C_n \prod_{j=1}^m \left[ D_j \int_{K_j} f_j \right]^p\]
as desired.
Because each $K_j$ is arbitrary, a second application of Monotone Convergence establishes the proposition.
 \end{proof}
 
 \subsection{Deduction of Theorem \ref{kakeyathm} from Proposition \ref{zkprop2}}
 \label{lastksec}
 \begin{proof}[Proof of Theorem \ref{kakeyathm}]
 As already observed, it suffices to assume each $f_j$ is smooth and compactly supported. As the submanifolds $\ri{y}$ depend smoothly on $y$, it follows from Proposition \ref{zkprop2} that for any $\delta > 0$,
  \[ \begin{aligned} 
 \int \left[ \int \cdots \int\prod_{j=1}^m f_j(y_j) \omega_x (\{ \delta^{-1} \ri{y_j} \}_{j=1}^m) dy_1 \cdots dy_m \right]^p dx
 \lesssim \prod_{j=1}^m \left( \int f_j \right)^{p}
 \end{aligned} \]
 for some implicit constant depending only on $n$ and the maximum degree of any $\ri{y}$.
After a change of variables $x \mapsto \delta^{-1} x$,
  \[ \begin{aligned} 
 \int \left[ \delta^{-m(n-k)} \int \cdots \int\prod_{j=1}^m f_j(y_j) \omega_{\delta^{-1} x} (\{ \delta^{-1} \ri{y_j} \}_{j=1}^m)  dy_1 \cdots dy_m \right]^p dx \\
 \lesssim \prod_{j=1}^m \left(\int f_j \right)^{p} 
 \end{aligned} \]
 uniformly for all positive $\delta$, where the factor $\delta^{-m(n-k)p} = \delta^{-n}$ arises as the Jacobian determinant of the change of variables.
 
 By Lemma \ref{specialdf} from the Appendix, it is possible to use an alternate defining function $\tilde \rho$ which exhibits better uniformity properties than $\rho$ itself might. In particular, for the defining function $\tilde \rho$ constructed there, the matrices $D_x \tilde \rho$ are exactly the orthogonal projections onto the orthogonal complement of the tangent space of $\ri{x}$ at $x$ and smallness of $|\tilde \rho(x,y)|$ implies proximity of $x$ to $\ri{y}$ in a uniform way: $|\tilde \rho(x,y)| \leq \delta \kappa_n$ for sufficiently small $\delta$ implies that the set $x + (-\delta,\delta)^n$ intersects $\ri{y}$ in a set of $k$-dimensional Hausdorff measure at least comparable to $\delta^k$. To proceed, one first observes that
 $T_{x_j} \delta^{-1} \ri{y_j}$ is the projection from $\R^n$ onto the orthogonal complement of the tangent space at $x_j \in \delta^{-1} \ri{y_j}$. By rescaling, the tangent plane of $\delta^{-1} \ri{y}$ at $x_j$ is simply a shift of the tangent plane at $\delta x_j$ of $\ri{y_j}$, so $T_{x_j} \delta^{-1} \ri{y_j} = D_{x_j} \tilde \rho(\delta x_j,y_j)$. Consequently, if $Q^{\delta}_x$ denotes the set $x + [-\delta/2,\delta/2]^n$, it follows that
 \begin{align*} 
  \omega_{\delta^{-1} x}  & ( \{ \delta^{-1} \ri{y_j} \}_{j=1}^m) \\ 
 = & \int_{\prod_{j=1}^m ((\delta^{-1}\ri{y_j}) \cap Q_{\delta^{-1}x})} \left[ \blw(\{T_{x_j} \delta^{-1} \ri{y_j} \}_{j=1}^m ) \right]^{\frac{1}{p}} d \mathcal H^{k}(x_1) \cdots d \mathcal H^k(x_m) \\
 \geq & \mathop{\inf_{x_1 \in (\delta^{-1} \ri{y_1}) \cap Q_{\delta^{-1}x},\ldots}}_{x_m \in (\delta^{-1} \ri{y_m}) \cap Q_{\delta^{-1}x}} [ \blw ( \{ D_{x_j} \tilde \rho(\delta x_j,y_j) \}_{j=1}^m)]^\frac{1}{p}  \prod_{j=1}^m {\mathcal H}^k \left( \delta^{-1} \ri{y_j} \cap Q_{\delta^{-1} x} \right) \\
 = & \mathop{\inf_{x_1 \in \ri{y_1} \cap Q^\delta_x,\ldots}}_{x_m \in \ri{y_m} \cap Q_x^\delta} [ \blw ( \{ D_{x_j} \tilde \rho(x_j,y_j) \}_{j=1}^m)]^\frac{1}{p}  \delta^{-km} \prod_{j=1}^m {\mathcal H}^k \left(\ri{y_j} \cap Q_{x}^\delta \right).
 \end{align*}
  By Lemma \ref{specialdf}, then, it follows that for any compact subset $K \subset \inc$, there is some open set $U$ containing $K$ such that whenever $\delta$ is sufficiently small,
  \begin{align*}\omega_{\delta^{-1} x} & ( \{ \delta^{-1} \ri{y_j} \}_{j=1}^m) \\ & \geq c_n^m \mathop{\inf_{x_1 \in \ri{y_1} \cap Q^\delta_x,\ldots}}_{x_m \in \ri{y_m} \cap Q_x^\delta} [ \blw ( \{ D_{x_j} \tilde \rho(x_j,y_j) \}_{j=1}^m)]^\frac{1}{p}  \prod_{j=1}^m \chi_{|\tilde \rho(x,y_j)| < \delta \kappa_n/2} \end{align*}
provided $(x,y_j) \in U$ for all $j=1,\ldots,m$. 

 Now the coarea formula dictates that for any continuous function $f_j$
 \[ \int f(y_j) \chi_{|\tilde \rho(x,y_j)| < \delta \kappa_n/2} dy_j = \int_{[-\delta \kappa_n/2,\delta \kappa_n/2]^{n-k}} \int_{\tilde \rho(x,\cdot) = u} f_j(y_j) d \sigma_u(y_j) d u, \]
where $d \sigma_u$ is a measure of continuous density with respect to $k$-dimensional Hausdorff measure on the level set $\set{y_j \in \R^n}{\tilde \rho(x,y_j) = u}$, which is a well-defined $k$-dimensional submanifold of $\R^n$ when $u$ is sufficiently small. In the special case $u = 0$, $ \sigma_0$ is exactly the measure $d \sigma$ on $\li{x}$ which was defined in \eqref{themeasure} (assuming that $\rho$ there is replaced by $\tilde \rho$). Since everything is continuous as a function of $\delta$ when $f$ is assumed to be continuous with compact support, the limit as $\delta \rightarrow 0^+$ of the quantity
\begin{align*}  \delta^{-m(n-k)} \int \cdots \int \prod_{j=1}^m f_j(y_j) \mathop{\inf_{x_1 \in \ri{y_1} \cap Q_{x}^\delta,\ldots}}_{x_m \in \ri{y_m} \cap Q_{x}^\delta}  [ \blw ( \{ D_{x_j} \tilde \rho(x_j,y_j) \}_{j=1}^m)]^\frac{1}{p} \\ \cdot \prod_{j=1}^m \chi_{|\tilde \rho(x,y_j)| < \delta \kappa_n/2} dy_1 \cdots d y_m 
\end{align*}
exists and equals a constant times
\[ \int_{\li{x}} \cdots \int_{\li{x}} [ \blw (\{D_{x} \tilde \rho(x,y_j) \}_{j=1}^m)]^{\frac{1}{p}} \prod_{j=1}^m f_j(y_j) d \sigma(y_1) \cdots d \sigma(y_m). \]
Thus
\begin{align*}
\int & \left[ \int_{\li{x}} \cdots \int_{\li{x}} [ \blw (\{D_{x} \tilde \rho(x,y_j) \}_{j=1}^m)]^{\frac{1}{p}} \prod_{j=1}^m f_j(y_j) d \sigma(y_1) \cdots d \sigma(y_m) \right]^p dx \\
 \leq & \limsup_{\delta \rightarrow 0^+} \int \left[ \frac{1}{\delta^{m(n-k)}} \int \cdots \int \omega_{\delta^{-1} x} ( \{ \delta^{-1} \ri{y_j}\}_{j=1}^m) \prod_{j=1}^m f_j(y_j) d y_1 \cdots d y_m \right]^p dx \\
 \lesssim & \prod_{j=1}^m \left( \int f_j \right)^p,
\end{align*}
which is 
the desired inequality \eqref{kakeya} with $\rho$ replaced by $\tilde \rho$.

To revert from $\tilde \rho$ back to $\rho$, it simply remains to assume that switching the defining function in this way leaves the left-hand side of \eqref{kakeya} unchanged.
This follows from the identify
\[ \left[ \blw( \{ M_j \pi_j \}_{j=1}^m ) \right]^{\frac{1}{p}} = \left[ \blw ( \{\pi_j\}_{j=1}^m) \right]^{\frac{1}{p}} \prod_{j=1}^m |\det M_j| \]
for Brascamp-Lieb constants, where $M_j$ are any invertible matrices. The inequality is easily proved by replacing each $f_j(u)$ with $f_j(M_j u)$ in \eqref{datadef}. Since $\tilde \rho$ differs from any fixed defining function $\rho$ by multiplication on the left by an invertible matrix, it follows by Lemma \ref{specialdf} that
\[ [ \blw(\{D_x \rho(x,y_j)\}_{j=1}^m) ]^{\frac{1}{p}} d \sigma(y_1) \cdots d \sigma (y_n) \]
is unchanged when defined using $\tilde \rho$ instead of $\rho$ itself because the extra factors of $\det ( D_x \rho (D_x \rho)^T)$ arising from the Brascamp-Lieb constant are exactly cancelled by the extra factors arising from the measure $d \sigma$. This completes the proof.
\end{proof}

 \section{The Brascamp-Lieb constant and Geometric Invariant Theory}
 \label{blsec}
The next major task is to establish several general facts about the Brascamp-Lieb constant and its connection to Geometric Invariant Theory. These facts play a central role in understanding and verifying the main hypothesis \eqref{nonconcentrate} of Theorem \ref{genradonthm}.  Throughout this section, for each $j=1,\ldots,m$,  each $\pi_j : \R^n \rightarrow \R^{n_j}$ will be an arbitrary linear map and each $p_j$ will be a real number in $[0,1]$. Following the usual convention, let the Brascamp-Lieb constant $\mathrm{BL}(\{\pi_j, p_j\}_{j=1}^m)$ be defined to equal the smallest nonnegative real number such that
 \begin{equation} \int_{\R^n} \prod_{j=1}^m \left( f_j ( \pi_j (x)) \right)^{p_j} dx \leq \mathrm{BL}(\{\pi_j, p_j\}_{j=1}^m) \prod_{j=1}^m \left( \int_{\R^{n_j}} f_j \right)^{p_j} \label{blineq} \end{equation}
for all nonnegative measurable functions $f_j \in L^1(\R^{n_j})$. When $p_1 = \cdots = p_m = \frac{n}{m(n-k)}$ and $n_1 = \cdots = n_m = n-k$, note that the Brascamp-Lieb constant is merely the reciprocal of the already-defined Brascamp-Lieb weight \eqref{datadef}. This special case will of course be the most important one for the purposes of Theorem \ref{genradonthm}, but throughout most of the section the $p_j$'s will be allowed to differ.

The overall goal of this section is to establish the existence of certain invariant polynomials in the entries of the $\pi_j$'s which give meaningful quantitative information about the Brascamp-Lieb constant. These polynomials should be thought of as generalizations of the determinant. For this description to be useful, it will be critical to show not only existence of such polynomials, but also to provide a means by which they may be explicitly constructed, so that they can be used as computational tools.
 
 \subsection{Brascamp-Lieb and minimum vectors}
 
 The first major result of this section is the following lemma, which establishes an identity for the Brascamp-Lieb constant involving an infimum analogous to the one relating to minimum vectors in the sense of Kempf and Ness \cite{kn1979}:
 \begin{lemma}
 Suppose that the exponents $p_j$ and dimensions $n_j$ satisfy \label{bl2mv}
 \begin{equation}  \sum_{j=1}^m \frac{p_j n_j}{n} = 1. \label{blscale} \end{equation}
 (Note: it is well-known and can be seen from scaling that \eqref{blscale} is necessary for the finiteness of the Brascamp-Lieb constant.) 
Then $\mathrm{BL}(\{\pi_j, p_j\}_{j=1}^m)$ satisfies
 \begin{equation} \left[ \mathrm{BL}(\{\pi_j, p_j\}_{j=1}^m) \right]^{-1} = \mathop{\mathop{\inf_{A_1 \in \SL{n_1},\ldots,}}_{A_m \in \SL{n_m},}}_{A \in \SL{n}} \prod_{j=1}^m n_j^{-\frac{p_j n_j}{2}} |||A_j \pi_j A^*|||^{p_j n_j},\label{minimumvector} \end{equation}
 where $||| \cdot |||$ denotes the Hilbert-Schmidt norm computed with respect to the standard bases and $\SL{n_j}$ is the Lie group of invertible $n_j \times n_j$ real matrices with determinant $1$.
 \end{lemma}
 
Before proceeding to the proof, it is worth observing that the direct link between the computation of the Brascamp-Lieb constant and Geometric Invariant Theory given by \eqref{minimumvector} provides a rather immediate interpretation of the work of
Garg, Gurvits, Oliveira, and Wigderson \cite{ggow2017}. Geometric Brascamp-Lieb data as they define it is exactly the set of data which are critical points of the functional on the right-hand side of \eqref{minimumvector} when $A_1,\ldots,A_n,A$ are all identity matrices (i.e., geometric Brascamp-Lieb data correspond to minimum vectors in GIT). The functional can be shown to be convex along flows $(A_1,\ldots,A_m,A) := (\exp(t M_1),\ldots,\exp(t M_m), \exp(t M))$, $t \in \R$, so critical points are automatically global minima. The iterative method in \cite{ggow2017} to compute the Brascamp-Lieb constant approximates  the argument of the infimum (argmin) of \eqref{minimumvector} when it exists by alternately computing the argmin $(A_1,\ldots,A_m)$ for fixed $A$ in one step and the argmin $A$ for fixed $(A_1,\ldots,A_m)$ in the subsequent step. (Also note that when the data is merely semi-stable and no global minimum exists, the algorithm instead produces a minimum vector with closed orbit contained in the original non-closed orbit.)

 \begin{proof}
 Lieb \cite{lieb1990} established that any Brascamp-Lieb inequality has an extremizing sequence of Gaussians, which implies that
\[ \left[ \mathrm{BL}(\{\pi_j, p_j\}_{j=1}^m) \right]^{-1} = \mathop{\inf_{A_1 \in \GL{n_1},\ldots,}}_{A_m \in \GL{n_m}}  \left[ \frac{ \det \left(  \sum_{j=1}^m p_j \pi_j^* A_j^* A_j \pi_j \right)}{ \prod_{j=1}^m ( \det A_j^* A_j )^{p_j}} \right]^\frac{1}{2}. \]
For any matrix $A \in \SL{n}$,
\[ \sum_{j=1}^m p_j ||| A_j \pi_j A^* |||^2 = \sum_{j=1}^m p_j \tr (A \pi_j^* A_j^* A_j \pi_j A^*) = \tr \left( \sum_{j=1}^m p_j A \pi_j^* A_j^* A_j \pi_j A^* \right). \]
Both the trace and determinant of the matrix $\sum_{j=1}^m p_j A \pi_j^* A_j^*A_j \pi_j A^*$ can be expressed in terms of its eigenvalues, all of which are nonnegative. By the inequality of arithmetic and geometric means, abbreviated as the AM-GM inequality, applied to the eigenvalues, it follows that
\[ \left| \sum_{j=1}^m \frac{p_j}{n} ||| A_j \pi_j A^*|||^2 \right|^n \! \! \geq  \det  \sum_{j=1}^m p_j A \pi_j^* A_j^* A_j \pi_j A^*   = \det \sum_{j=1}^m p_j \pi_j^* A_j^* A_j \pi_j . \]
When the infimum of the left-hand side is taken over all $A \in \SL{n}$, the inequality must be equality; to see this, fix $M := \sum_{j=1}^m p_j \pi_j^* A_j^* A_j \pi_j$. When $M$ is invertible, equality must hold when $A := M^{-1/2} (\det M)^{1/(2n)}$; if $M$ has a kernel of dimension $\ell > 0$, let $P$ be orthogonal projection onto the kernel. Equality holds in the limit $t \rightarrow \infty$ when $A_t := t^{1/\ell} P + t^{-1/(n-\ell)} (I - P)$. Therefore
\[ \left[ \mathrm{BL}(\{\pi_j, p_j\}_{j=1}^m) \right]^{-1} = \mathop{\mathop{\inf_{A_1 \in \GL{n_1},\ldots,}}_{A_m \in \GL{n_m},}}_{A \in \SL{n}} \left[ \frac{  \sum_{j=1}^m p_j ||| A_j \pi_j A^*|||^2 }{ n \prod_{j=1}^m |\det A_j|^{\frac{2p_j}{n}} } \right]^{\frac{n}{2}}. \]
A similar application of the AM-GM inequality also gives that
\[\mathop{\inf_{t_1 > 0, \ldots,}}_{t_m > 0}  t_1^{-\frac{2 p_1 n_1}{n}} \cdots t_m^{-\frac{2p_m n_m}{n}} \sum_{j=1}^m \frac{p_j n_j}{n}  t_j^2 \frac{ ||| A_j \pi_j A^* |||^2}{n_j} =  \prod_{j=1}^m \left( \frac{|||A_j \pi_j A^* |||^2}{n_j}\right)^{\frac{p_j n_j}{n}}. \]
To see this, the left-hand side can be seen to be greater than or equal to the right-hand side by using the version of AM-GM inequality which raises the term
$$t_1^{-\frac{2 p_1 n_1}{n}} \cdots t_m^{-\frac{2p_m n_m}{n}} t_j^2 \frac{||| A_j \pi_j A^*|||^2}{n_j}$$
to the power $p_j n_j/n$, which is allowed precisely because \eqref{blscale} guarantees that the exponents sum to $1$. The reverse inequality can be established by fixing $t_j := ( ||| A_j \pi_j A^*|||^2/n_j )^{-1/2}$ when all such constants are well-defined or by an appropriate limiting argument if any such $t_j$ happens to be infinite. Writing each matrix $A_j$ as a nonzero constant times a matrix of determinant $1$ then gives that
\[ \left[ \mathrm{BL}(\{\pi_j, p_j\}_{j=1}^m) \right]^{-1} = \mathop{\mathop{\inf_{A_1 \in \SL{n_1},\ldots,}}_{A_m \in \SL{n_m},}}_{A \in \SL{n}} \prod_{j=1}^m n_j^{-\frac{p_j n_j}{2}} |||A_j \pi_j A^*|||^{p_j n_j}. \]
This is exactly \eqref{minimumvector}.
\end{proof}

Before continuing, it will be helpful record an important calculation relating to Lemma \ref{bl2mv} which will be useful later.
As it relates to the hypothesis \eqref{nonconcentrate} of Theorem \ref{genradonthm}, Lemma \ref{bl2mv} establishes that
\begin{equation} \left[ \blw(\{\pi_j\}_{j=1}^m) \right]^{\frac{1}{p}}  = \mathop{\inf_{A_1,\ldots,A_m \in \SL{n-k}}}_{A \in \SL{n}}  \frac{1}{(n-k)^{\frac{m(n-k)}{2}}  }  \prod_{j=1}^m ||| A_j \pi_j A^*|||^{n-k}  \label{minimumvector2} \end{equation}
when each $\pi_j$ is an $(n-k) \times n$ matrix and $1/p = m(n-k)/n$.

The next step in this section is to give an abstract proof of the existence of invariant polynomials in the entries of the $\pi_j$'s which strongly quantify the magnitude of the Brascamp-Lieb constant. Following this, we will consider the question of how to more explicitly find these polynomials.

A few minor reductions are in order. The first is that attention will be restricted to only those cases in which each $p_j$ is rational. By Theorem 1.13 of Bennett, Carbery, Christ, and Tao \cite{bcct2008}, the extreme points of the convex set 
\[ P := \set{ \{p_j\}_{j=1}^m \in [0,1]^m}{ \mathrm{BL}(\{\pi_j, p_j\}_{j=1}^m) < \infty} \]
all have rational exponents $\{p_j\}_{j=1}^m$, and likewise rational exponents play a central role in Theorem \ref{genradonthm}.
It may also be assumed that no $p_j$ equals zero since the inequality \eqref{blineq} will be trivially independent of $\pi_j$ for any index $j$ such that $p_j = 0$, meaning that one can simply reduce $m$ and consider the Brascamp-Lieb inequality for a strictly smaller number of $\pi_j$'s.

The expression \eqref{minimumvector} has deep connections to the theory of minimum vectors in Geometric Invariant Theory. Pursuing this analogy, it is natural to make a connection between $\mathrm{BL}(\{\pi_j,p_j\}_{j=1}^m)$ and polynomials invariant under the underlying group representation
$\rho$ of $\SL{n_1} \times \cdots \times \SL{n_m} \times \SL{n}$ defined by
\begin{equation} \rho_{(A_1,\ldots,A_m,A)} (\{ \pi_j \}_{j=1}^m) := \{ A_j \pi_j A^* \}_{j=1}^m. \label{blgroup} \end{equation}
Let $\Phi$ be any nonzero polynomial function of the matrices $\{\pi_j\}_{j=1}^m$ which is homogeneous of degree $d_j > 0$ in each $\pi_j$ and is $\rho$-invariant, i.e.,
\begin{equation} \Phi( \{ \lambda_j \pi_j \}_{j=1}^m ) = \lambda_1^{d_1} \cdots \lambda_m^{d_m} \Phi ( \{ \pi_j \}_{j=1}^m) \mbox{ for all } \lambda_1,\ldots,\lambda_m \in \R \label{algebra1} \end{equation}
and
\begin{equation} \Phi (\{ A_j \pi_j A^* \}_{j=1}^m) = \Phi ( \{\pi_j\}_{j=1}^m ) \label{algebra2} \end{equation}
whenever $\det A_1 = \cdots = \det A_m = 1 = \det A$.
If $||| \Phi |||$ is the maximum of $|\Phi|$ on all $m$-tuples $\{\tilde \pi_j\}_{j=1}^m$ such that $||| \tilde \pi_j ||| \leq 1$ for all $j=1,\ldots,m$, then scaling dictates that
\[ ||| \Phi ||| \prod_{j=1}^m ||| A_j \pi_j A^* |||^{d_j} \geq | \Phi( \{A_j \pi_j A^* \}_{j=1}^m) | =  | \Phi( \{ \pi_j \}_{j=1}^m ) | \]
for all inputs $\{\pi_j\}_{j=1}^m$. If each degree $d_j$ happens to satisfy
\begin{equation} \frac{p_1 n_1}{d_1} = \cdots = \frac{p_m n_m}{d_m} = \frac{1}{s_\Phi} \label{algebra3} \end{equation}
for some real number $s_\Phi$, then \eqref{minimumvector} implies that
\begin{equation}  \left[ \mathrm{BL}(\{\pi_j, p_j\}_{j=1}^m) \right]^{-1}\geq \left( \prod_{j=1}^m n_j^{-\frac{p_j n_j}{2}} \right) ||| \Phi|||^{-\frac{1}{s_\Phi}} | \Phi(\{\pi_j\}_{j=1}^m) |^{\frac{1}{s_\Phi}}. \label{blblock} \end{equation}
In the specific case relating to Theorem \ref{genradonthm}, the constraint \eqref{algebra3} is trivially satisfied whenever $d_1 = \cdots = d_m$ and \eqref{minimumvector2} yields the inequality
\begin{equation}   \left[ W (\{ \pi_j \}_{j=1}^m) \right]^{\frac{1}{p}}   \geq (n-k)^{-\frac{m(n-k)}{2}} ||| \Phi|||^{-\frac{n-k}{d}} | \Phi(\{\pi_j\}_{j=1}^m) |^{\frac{n-k}{d}}. \label{blblock2} \end{equation}
The following lemma establishes that the collection of all such invariant polynomials can be used to compute the order of magnitude of the Brascamp-Lieb constant:
\begin{lemma}
Suppose that the exponents $\{p_j\}_{j=1}^m \in (0,1]^m$ are rational and satisfy \eqref{blscale}. Let $\mathrm{IP}$ be the collection of all nonzero invariant polynomials $\Phi$ satisfying \eqref{algebra1}, \eqref{algebra2}, and \eqref{algebra3}. Then \label{bl2polylem}
\begin{equation} \left[ \mathrm{BL}(\{\pi_j, p_j\}_{j=1}^m) \right]^{-1} \approx \sup_{\Phi \in \mathrm{IP}} ||| \Phi|||^{-\frac{1}{s_\Phi}} | \Phi(\{\pi_j\}_{j=1}^m) |^{\frac{1}{s_\Phi}} \label{blcompare} \end{equation}
with implicit constants that are independent of $\{\pi_j\}_{j=1}^m$ (where the supremum is understood to be zero if $\mathrm{IP} = \emptyset$). Moreover, there exists a finite subset $\mathrm{IP}_0 \subset \mathrm{IP}$ such that
\[ \sup_{\Phi \in \mathrm{IP}} |||\Phi|||^{-\frac{1}{s_\Phi}} | \Phi( \{\pi_j\}_{j=1}^m)|^{\frac{1}{s_\Phi}} \approx \sup_{\Phi \in \mathrm{IP}_0} |||\Phi|||^{-\frac{1}{s_\Phi}} | \Phi( \{\pi_j\}_{j=1}^m)|^{\frac{1}{s_\Phi}}. \]
\end{lemma}
\begin{proof}
The lower bound follows immediately from \eqref{blblock}. The upper bound will be proved by contradiction. Without loss of generality, it may be assumed that data exists such that the left-hand side of \eqref{blcompare} is strictly positive. Suppose for each positive integer $N$, there is some data $\{\pi^N_j\}_{j=1}^m$ such that
\begin{equation} \left[ \mathrm{BL}(\{\pi_j^N, p_j\}_{j=1}^m) \right]^{-1} \geq N \sup_{\Phi \in \mathrm{IP}} |||\Phi|||^{-\frac{1}{s_\Phi}} | \Phi( \{\pi_j^N\}_{j=1}^m)|^{\frac{1}{s_\Phi}}. \label{cont} \end{equation}
By homogeneity of both sides in the data $\{ \pi_j^N \}_{j=1}^m$, it may be assumed that $\mathrm{BL}(\{\pi_j^N, p_j\}_{j=1}^m) = 1$ for each $N$, and by replacing each tuple $\{\pi_j^N\}_{j=1}^m$ with $\rho_{(A_1^N,\ldots,A_m^N,A^N)} (\{\pi_j^N\}_{j=1}^m)$ for some choice of $A_1^N,\ldots,A_m^N$ and $A^N$ for each $N$ which tend to minimizers of the right-hand side of \eqref{minimumvector} as $N \rightarrow \infty$, it may further be assumed that
\[ \prod_{j=1}^m ||| \pi_j^N |||^{p_j n_j} \rightarrow \prod_{j=1}^m n_j^{\frac{p_j n_j}{2}} \]
as $N \rightarrow \infty$. Once again, noting that both sides of \eqref{cont} are homogeneous in $\pi_j$ for each $j$, rescaling individual $\pi_j$'s as necessary allows one to assume that $||| \pi_j^N ||| \rightarrow n_j^{1/2}$ as $N \rightarrow \infty$ for each $j = 1,\ldots,m$. By passing to a subsequence in $N$, this means that $\pi_j^N$ converges to some limiting data for each $j=1,\ldots,m$. Let this limit data be denoted $\{ \pi_j^\infty\}_{j=1}^m$.   Now for any matrices $A_1,\ldots,A_m, A$, by Lemma \ref{bl2mv},
\begin{align*}
 \prod_{j=1}^m & ||| A_j \pi_j^\infty A^* |||^{p_j n_j} = \lim_{N \rightarrow \infty}  \prod_{j=1}^m ||| A_j \pi_j^N A^* |||^{p_j n_j } \\
 & \geq \limsup_{N \rightarrow \infty} \left( \prod_{j=1}^m n_j^{\frac{p_j n_j}{2}} \right) \left[ \mathrm{BL}(\{\pi_j^N, p_j\}_{j=1}^m) \right]^{-1} = \prod_{j=1}^m n_j^{\frac{p_j n_j}{2}}, 
 \end{align*}
so taking an infimum over all $A_1,\ldots,A_m,A$ gives that $\mathrm{BL}(\{\pi_j^\infty, p_j\}_{j=1}^m) \leq 1$. In fact, this inequality must be an equality, which can be seen by simply taking each $A_j$ and $A$ to be the identity. Now for any $\Phi \in \mathrm{IP}$, 
\[ 1 =  \left[ \mathrm{BL}(\{\pi_j^N, p_j\}_{j=1}^m) \right]^{-1} \geq N |||\Phi|||^{-\frac{1}{s_\Phi}} | \Phi( \{\pi_j^N\}_{j=1}^m)|^{\frac{1}{s_\Phi}}, \]
which means that $\Phi(\{\pi_j^N\}_{j=1}^m) \rightarrow 0$ as $N \rightarrow \infty$. By continuity of each $\Phi$, it follows that 
\begin{equation} \left[ \mathrm{BL}(\{\pi_j^\infty, p_j\}_{j=1}^m) \right]^{-1} = 1 \mbox{ and }  \sup_{\Phi \in \mathrm{IP}} |||\Phi|||^{-\frac{1}{s_\Phi}} | \Phi( \{\pi_j^\infty \}_{j=1}^m)|^{\frac{1}{s_\Phi}} = 0. \label{cbackhere} \end{equation}

Since each exponent $p_j$ is rational and nonzero, it must be possible to find positive integers $q_1,\ldots,q_m$ and $q$ such that $p_j n_j = q_j / q$ for each $j$.
Now suppose  that 
\[ \Pi ( \{x_i^1,y_i^1\}_{i=1}^{q_1},\ldots, \{x_i^m,y_i^m\}_{i=1}^{q_m} ) \]
is any real-valued map which is linear in each $x_i^{j} \in \R^{n_j}$ and each $y_i^{j} \in \R^n$ for $i = 1,\ldots,q_j$ and $j=1,\ldots,m$.
The group $\SL{n_1} \times \cdots \times \SL{n_m} \times \SL{n}$ acts on the vector space $V$ of all such $\Pi$ by defining
\begin{align*} \rho_{(A_1,\ldots,A_m,A)} & \Pi ( \{x_i^1,y_i^1\}_{i=1}^{q_1},\ldots, \{x_i^m,y_i^m\}_{i=1}^{q_m} ) \\ & := \Pi ( \{A_1^* x_i^1,A^* y_i^1\}_{i=1}^{q_1},\ldots, \{A_m^* x_i^m,A^* y_i^m\}_{i=1}^{q_m} ).
\end{align*}
Let $\Pi^\infty \in V$ be the multilinear functional given by
\begin{equation} 
\Pi^\infty ( \{x_i^1,y_i^1\}_{i=1}^{q_1},\ldots, \{x_i^m,y_i^m\}_{i=1}^{q_m} ) := \prod_{j=1}^m \prod_{i=1}^{q_j} \ang{ x_i , \pi_j^\infty y_i}
\label{restrpr} \end{equation}
where $\ang{ \cdot , \cdot}$ is the usual inner product on $\R^n$.  The Hilbert-Schmidt norm of $\rho_{(A_1,\ldots,A_m,A)} \Pi^\infty$ is exactly equal to
\[ \prod_{j=1}^m ||| A_j \pi_j^\infty A_j^* |||^{q_j}, \]
so by Lemma \ref{bl2mv}, it follows that
\begin{equation} 1 = \left[ \mathrm{BL}(\{\pi_j^\infty,p_j\}_{j=1}^m) \right]^{-1} = \prod_{j=1}^m n_j^{-\frac{q p_j n_j}{2}} \mathop{\mathop{\inf_{A_1 \in \SL{n_1},\ldots,}}_{A_m \in \SL{n_m},}}_{A \in \SL{n}}  ||| \rho_{(A_1,\ldots,A_m,A)} \Pi^\infty |||. \label{not0closed} \end{equation}
By the real Hilbert-Mumford criterion \cite{birkes1971}, $0$ belongs to the closure of the $\rho$-orbit of $\Pi^\infty$ in the standard topology if and only if $0$ belongs to the Zariski closure; furthermore, $0$ belongs to the Zariski closure if and only if all nonconstant homogeneous $\rho$-invariant polynomials on $V$ vanish on $\Pi^\infty$. Since \eqref{not0closed} guarantees that $0$ is not in the standard closure of the orbit, there must exist a nonconstant homogeneous $\rho$-invariant polynomial $P$ on $V$ such that $P(\Pi^\infty) \neq 0$. If the degree of $P$ is equal to $d$, then $P(\Pi^\infty)$ must itself be a polynomial function of $\{\pi_j\}_{j=1}^m$ which satisfies both \eqref{algebra1} (with $d_j := d q_j$ for each $j$) and \eqref{algebra2}. This means that $P(\Pi^\infty)$ also satisfies \eqref{algebra3} with $s_\Phi := d q$. Thus this polynomial $P(\Pi^\infty)$ contradicts \eqref{cbackhere}.

The the finite subset $\mathrm{IP}_0$ can be taken to be only those polynomials of the form $P(\Pi^\infty)$ for $P$ belonging to any finite generating set of the $\rho$-invariant algebra on $V$, since the contradiction just derived will still hold if $P(\Pi^\infty) = 0$ for all such polynomials.
\end{proof}

\subsection{Invariant polynomials and the Caley $\Omega$ process}

While Lemma \ref{bl2polylem} a the theoretical foundation upon which much of this paper rests, it is necessary to have a more concrete way of describing polynomials in the class $\mathrm{IP}$. To that end, it is useful to appeal to the very old and
well-known fact in invariant theory that invariants associated to the group $\SL{n}$ are generated by application of the ``Cayley $\Omega$ process,'' which is briefly described here as it applies to the more general situation of Brascamp-Lieb invariant polynomials satisfying \eqref{algebra1}, \eqref{algebra2} and \eqref{algebra3}. As before, it will be assumed that the exponents $p_j$ are positive, rational, and satisfy the scaling condition \eqref{blscale}.

If $\Phi$ is any polynomial in $\{\pi_j\}_{j=1}^m$ satisfying \eqref{algebra1}, \eqref{algebra2} and \eqref{algebra3}, then for any matrices $A_1,\ldots,A_m, A$ with strictly positive determinants, by homogeneity and $\rho$-invariance it must be the case that
\begin{equation} \Phi ( \{A_j \pi_j A^*\}_{j=1}^m) = \left[ (\det A)^{s_\Phi} \prod_{j=1}^m (\det A_j)^{p_j s_\Phi} \right] \Phi(\{\pi_j\}_{j=1}^m). \label{detid} \end{equation}
Since matrices with positive determinant form an open set in $\R^{n\times n}$ for all $n$ and since the left-hand side of the identity \eqref{detid} must be a polynomial function in the entries of each $A_j$, this forces $s_\Phi$ and $p_j s_\Phi$ to be positive integers and it further forces \eqref{detid} to hold for all matrices $A_1,\ldots, A_m$ and $A$ even if some of the determinants are zero or negative. 

Let $\Omega_A$ be the Cayley $\Omega$ operator associated to $A$, i.e.,
\[ \Omega_A := \sum_{\sigma \in {\mathfrak S}_n} (-1)^{\sigma} \frac{\partial}{\partial A_{1 \sigma_1}} \cdots \frac{\partial}{\partial A_{n \sigma_n}}. \]
(Here and throughout the remainder of Section \ref{blsec}, $\sigma$ will denote a permutation rather than referring to the measure \eqref{themeasure}.)
The Cayley $\Omega$ operator associated to $A$ satisfies the identity
\[ \Omega^{s}_{A} (\det A)^s = c_{n,s} > 0 \]
for all positive integers $s$ and also satisfies $\Omega_A f( B A) = (\det B) (\Omega_A f) (B A)$ for any $n \times n$ matrix $B$ and any $C^n$ function $f$ of $\R^{n \times n}$ (for both facts, see Sturmfels \cite{sturmfelsbook}). These facts together imply that
\[ \Omega_A^{s_\Phi} \Omega_{A_1}^{p_1 s_\Phi} \cdots \Omega_{A_m}^{p_m s_\Phi} \Phi(\{A_j \pi_j A^*\}_{j=1}^m) = c \Phi(\{\pi_j\}_{j=1}^m) \]
for some nonzero constant $c$ depending only on the exponents $d_j$, $p_j$, and $n_j$ when $\Phi$ satisfies \eqref{algebra1}, \eqref{algebra2} and \eqref{algebra3}. They also imply that 
that for any $\Phi$ satisfying \eqref{algebra1} and \eqref{algebra3} only, the function of $\{\pi_j\}_{j=1}^m$ given by $$\Omega_A^{s_\Phi} \Omega_{A_1}^{p_1 s_\Phi} \cdots \Omega_{A_m}^{p_m s_\Phi} \Phi(\{A_j \pi_j A^*\}_{j=1}^m)$$ necessarily satisfies each of \eqref{algebra1}, \eqref{algebra2}, and \eqref{algebra3}. To understand the space of homogeneous invariant polynomials of a given degree, then, it suffices to understand the image of the map $\Phi \mapsto \Omega_A^{s_\Phi} \Omega_{A_1}^{p_1 s_\Phi} \cdots \Omega_{A_m}^{p_m s_\Phi} \Phi(\{A_j \pi_j A^*\}_{j=1}^m)$ for polynomials $\Phi$ satisfying \eqref{algebra1} and \eqref{algebra3} only.

\subsection{Polynomial invariants of Brascamp-Lieb data}

\label{blpoly}

We come now to the main result of this section, which gives a concrete characterization of the class $\mathrm{IP}$ in terms of polynomials which are expressible as determinants of block-form matrices. In light of Lemma \ref{bl2polylem}, these determinants can be reasonably regarded as quantifying various sorts of transversality of the maps $\{\pi_j\}_{j=1}^m$ which allow for finiteness of the Brascamp-Lieb constant for any desired rational exponents $\{p_j\}_{j=1}^m \in (0,1]^m$. This approach to understanding the Brascamp-Lieb constant is complementary to the work of Bennett, Carbery, and Tao \cite{bct2006} and Bennett, Carbery, Christ, and Tao \cite{bcct2008} in exactly the same way that direct computations with invariant polynomials complement characterizations of the nullcone in Geometric Invariant Theory. The strength of the finiteness criteria established in \cite{bcct2008} is that one need only show that a single (cleverly-chosen) inequality is violated to deduce that the Brascamp-Lieb constant is infinite.
Lemma \ref{bl2polylem}, in contrast, allows one to deduce the \textit{finiteness} of the constant by demonstrating the nonvanishing of a single (cleverly-chosen) invariant polynomial.

\begin{lemma}
Suppose $\{ p_j \}_{j=1}^m \in (0,1]^m$ be rational exponents satisfying the scaling condition \eqref{blscale}. Let $s$ be an integer such that $p_j s$ is an integer for all $j=1,\ldots,m$. Let $V_s$ be the vector space of all polynomials $\Phi$ satisfying \eqref{algebra1}, \eqref{algebra2}, and \eqref{algebra3} for $s_\Phi = s$. Then $V_s$ is spanned by polynomials of the form $\det M( \{\pi_j\}_{j=1}^m)$, where $M(\{ \pi_j\}_{j=1}^m)$ is an $ns \times ns$ matrix consisting of block elements of size $n_j \times n$ for $j=1,\ldots,m$ arranged in the following way: \label{itsadet}
\begin{itemize}
\item Each block entry is a constant multiple of $\pi_j$ for some $j = 1,\ldots,m$.
\item For each $j=1,\ldots,m$, there are $p_j s$ block rows of height $n_j$ (i.e., the block row is a group of $n_j$ adjacent rows of $M$). In each such block row, all block entries are multiples of $\pi_j$. At most $n_j$ of these block entries are nonzero.
\item There are $s$ block columns of width $n$. In each block column, there are at most $n$ nonzero block entries.
\end{itemize}
Figure \ref{scheme1} illustrates the structure of all such matrices $M$.
\end{lemma}
\begin{figure}[ht]
\begin{center}
\begin{tikzpicture}[rotate=90,xscale=-0.37,yscale=-0.37]
\draw[xstep=3,ystep=7,dashed] (0,0) grid (21,21);
\draw (0,0) rectangle (21,21);
\draw (0,0) rectangle (3,7);  \node at (1.5,3.5) {$\displaystyle c_{111} \pi_1$};  %UL
\draw (0,14) rectangle (3,21); \node at (1.5,17.5) {$\displaystyle c_{11s} \pi_1$}; %UR
\node at (1.5,10.5) {$\displaystyle \cdots$}; \node at (7.5,10.5) {$\displaystyle \cdots$};
\node at (4.5,3.5) {$\displaystyle \vdots$}; \node at (4.5,17.5) {$\displaystyle \vdots$};
\node at (4.5,10.5) {$\displaystyle \ddots$};
\draw (6,0) rectangle (9,7);  \node at (7.5,3.5) {$\displaystyle c_{1 (p_1 s)1} \pi_1$};  %UL
\draw (6,14) rectangle (9,21); \node at (7.5,17.5) {$\displaystyle c_{1 (p_1 s) s} \pi_1$}; %UR

\node at (10.5,3.5) {$\displaystyle \vdots$}; \node at (10.5,17.5) {$\displaystyle \vdots$};
\node at (10.5,10.5) {$\displaystyle \ddots$};

\draw (12,0) rectangle (15,7);  \node at (13.5,3.5) {$\displaystyle c_{m11}\pi_m$};  %UL
\draw (12,14) rectangle (15,21); \node at (13.5,17.5) {$\displaystyle c_{m1 s} \pi_m$}; %UR
\node at (13.5,10.5) {$\displaystyle \cdots$}; \node at (19.5,10.5) {$\displaystyle \cdots$};
\node at (16.5,3.5) {$\displaystyle \vdots$}; \node at (16.5,17.5) {$\displaystyle \vdots$};
\node at (16.5,10.5) {$\displaystyle \ddots$};
\draw (18,0) rectangle (21,7);  \node at (19.5,3.5) {$\displaystyle c_{m (p_m s) 1} \pi_m$};  %UL
\draw (18,14) rectangle (21,21); \node at (19.5,17.5) {$\displaystyle c_{m (p_m s) s} \pi_m$}; %UR

\draw [decorate,decoration={mirror,brace,amplitude=10pt}]
(0,0) -- (9,0) node [black,midway,left,text width=2.5cm] 
{$p_1 s$ block rows of height $n_1$};

\node at (10.5,-4) {$\displaystyle \vdots$}; 

\draw [decorate,decoration={mirror,brace,amplitude=10pt}]
(12,0) -- (21,0) node [black,midway,left,text width=2.5cm] 
{$p_m s$ block rows of height $n_m$};

\draw [decorate,decoration={brace,amplitude=10pt}]
(0,0) -- (0,21) node [black,midway,above,yshift=10pt] 
{$s$ block columns of width $n$};

\end{tikzpicture}
\end{center}
\caption{Block structure of $ns \times n s$ matrices $M$ whose determinants span the space of invariant polynomials of Brascamp-Lieb data satisfying \eqref{algebra1}, \eqref{algebra2}, and \eqref{algebra3} for $s_\Phi = s$.  Here each $c_{i_1i_2i_3}$ is a scalar. \label{scheme1} }
\end{figure}
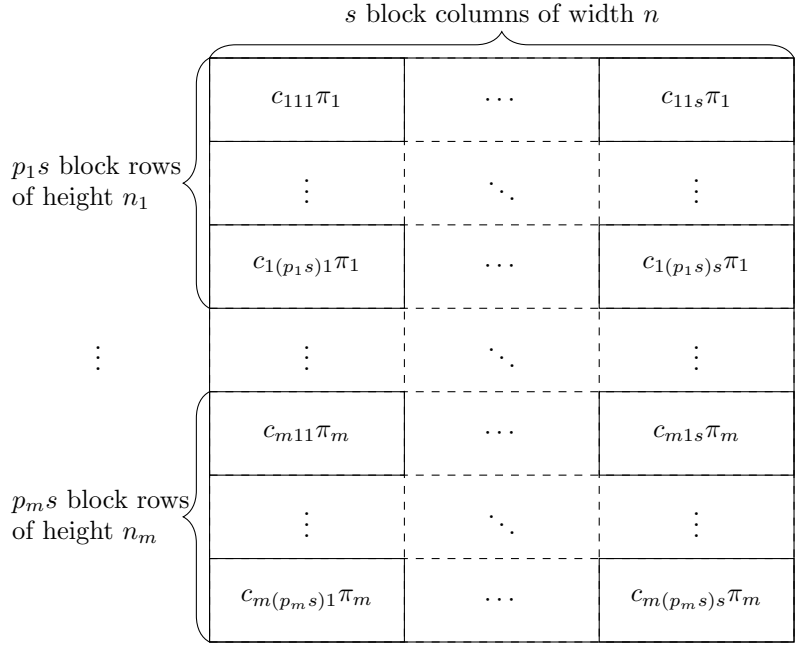

\begin{proof}
The proof proceeds by an analysis of the action of the Cayley $\Omega$ operator on general multilinear functionals.
One could instead formulate this problem as a quiver representation and appeal to a number of general results concerning the structure of semi-invariants (see, for example Domokos and Zubikov \cite{dz2001}), but for the present purposes the $\Omega$ operator will yield a more elementary and transparent proof from the standpoint of analysis. Readers should also note the similarity of the matrices $M(\{\pi_j\}_{j=1}^m)$ and the Brascamp-Lieb operator as defined in \cite{ggow2017}.

Suppose that $\Pi : (\R^n)^{n} \rightarrow \R$ is a multilinear functional on $\R^n$. This $\Pi$ is expressed in the standard basis by the formula
\[ \Pi(\{x_i\}_{i=1}^{n}) := \sum_{j_1,\ldots,j_{n}=1}^n \Pi_{j_1 \cdots j_{n}} x_{1, j_1} \cdots x_{n, j_{n}} \]
where $x_{i,j}$ is the $j$-th coordinate of $x_{i}$.
For any $n \times n$ matrix $A$, 
\[ \Pi (\{A x_i\}_{i=1}^{n}) = \mathop{\sum_{j_1,\ldots,j_{n}=1}^n}_{k_1,\ldots,k_{n}=1} \Pi_{j_1 \cdots j_{n}} A_{j_1 k_1} \cdots A_{j_{n} k_{n}} x_{1, k_1} \cdots x_{n, k_{n}}. \]
If this sum is differentiated by ${\partial^n}/{\partial A_{1\sigma_1} \cdots \partial A_{n \sigma_n}}$, the result will equal zero unless $j_1,\ldots,j_n$ are distinct and $k_i = \sigma_{j_i}$ for each $i=1,\ldots,n$. Thus
\begin{equation} \frac{\partial^n}{\partial A_{1\sigma_1} \cdots \partial A_{n \sigma_n}} \Pi (\{A x_i\}_{i=1}^{n}) = \mathop{\sum_{j_1,\ldots,j_n = 1}^n}_{\mbox{distinct}} \Pi_{j_1 \cdots j_n} x_{1, \sigma_{j_1}} \cdots x_{n, \sigma_{j_n}}. \label{almostc1} \end{equation}
%The Cayley $\Omega$ operator is exactly the differential operator
%\[ \Omega_A := \sum_{\sigma \in {\mathfrak S}_n} \partial_{A_{1 \sigma_1} \cdots A_{n \sigma_n}}; \]
Multiplying \eqref{almostc1} by $(-1)^{\sigma}$ and summing over $\sigma \in {\mathfrak S}_n$ gives that
\begin{align*}
\Omega_A \Pi (\{A x_i\}_{i=1}^{n}) & = \left( \sum_{\tau \in {\mathfrak S}_n} (-1)^{\tau} \Pi_{\tau_1 \cdots \tau_{n}} \right) \left( \sum_{\sigma  \in {\mathfrak S}_n} (-1)^{\sigma} x_{1, \sigma_1} \cdots x_{n, \sigma_n} \right) \\
& = \left( \sum_{\tau \in {\mathfrak S}_n} (-1)^{\tau} \Pi_{\tau_1 \cdots \tau_{n}} \right)  [ x_1 \cdots x_n ].
\end{align*}
The notation $[x_1 \cdots x_n]$ is simply shorthand for the determinant of the $n\times n$ matrix whose columns are given by the vectors $x_1,\ldots,x_n$.
The quantity in parentheses on the last line above will be called the alternating contraction of $\Pi$ in the indices $(1,\ldots,n)$ and will be denoted $\left. \Pi \right|_{(1,\ldots,n)}$. Suppose now that $\Pi$ has some arbitrary degree of multilinearity, i.e., $\Pi : (\R^n)^{\Lambda} \rightarrow \R$ for some ordered index set $\Lambda$. If $\#\Lambda < n$, then $\Omega_A \Pi (\{A x_i\}_{i \in \Lambda}) = 0$ trivially. If instead $k > n$, then by the product rule it must be the case that
\[ \Omega_A \Pi (\{A x_i\}_{i \in \Lambda}) = \mathop{\sum_{I \subset \Lambda}}_{\# I = n} \left. \Pi \right|_{I}  ( \{ A x_i\}_{i \in \Lambda \setminus I}) [ x ]_{I} \]
where $\left. \Pi \right|_I$ is the multilinear functional with index set $\Lambda \setminus I$ obtained by performing an alternating contraction in the indices $I$ (arranged in the usual order) and where $[x]_{I} := [x_{i_1} \cdots x_{i_n}]$ with $i_1 < \cdots < i_n$ being the elements of $I$. By induction, for any $s$ such that $\# \Lambda \geq ns$,
\begin{equation} \begin{split} \Omega_A^{s}&  \Pi (\{A x_i\}_{i \in \Lambda}) \\ & = \mathop{ \sum_{\# I_1 = n} \cdots \sum_{\# I_s = n} }_{I_1,\ldots I_s \mbox{ pairwise disjoint}} \left.   \Pi \right|_{I_1} \cdots  \left. \right|_{I_s} ( \{ A x_i\}_{i \in \Lambda \setminus \bigcup_{j=1}^s I_j )}) [x]_{I_1} \cdots [x]_{I_s}. \end{split} \label{expandit} \end{equation}
When $\# \Lambda = ns$ and $\Lambda = I_1 \cup \cdots \cup I_s$ for pairwise disjoint $I_j$'s, the quantity $\left. \Pi \right|_{I_1} \cdots  \left. \right|_{I_s}$ is simply a scalar obtained by performing an alternating contraction in each of the index subsets $I_1,\ldots,I_s$.

Now consider the multilinear functional  
\begin{equation} \Pi( \{x_i^1,y_i^1\}_{i=1}^{q_1},\ldots, \{x_i^m,y_i^m\}_{i=1}^{q_m} ) := \prod_{j=1}^m \prod_{i=1}^{q_j} \ang{ x_j^i, \pi_j y_j^i} \label{mytensor1} \end{equation}
where $p_j n_j = q_j/q$ and where the $\pi_j$ are as in the previous section; this is exactly the same construction as \eqref{restrpr}. If $A \in \SL{n}$ and $A_j \in \SL{n_j}$ for each $j=1,\ldots,m$, then we seek homogeneous polynomials of degree $d$ in the entries of $\Pi$ which are invariant under the action of these matrices given by
\[\Pi( \{A_1 x_i^1,A y_i^1\}_{i=1}^{q_1},\ldots, \{A_m x_i^m,A y_i^m\}_{i=1}^{q_m} ). \]
(Note that this action differs from $\rho$ by replacing $A_j^*$ and $A^*$ by $A_j$ and $A$; since the special linear group is closed under adjoints, this change is inconsequential and simplifies notation.)
Any polynomial function of $\Pi$ must belong to the span of $d$-fold products of the expressions \eqref{mytensor1}, where in each term of the product, the $x_i^j$'s and $y_i^j$'s are regarded as fixed but may change from factor to factor (which is to say that evaluating $\Pi$ on specific tuples of $x_i^j$'s and $y_i^j$'s gives a basis of functions from which the algebra of polynomial functions of $\Pi$ can be generated). If this polynomial happens to be invariant under the action of the matrices $(A_1,\ldots,A_m,A)$, that polynomial must be preserved (up to multiplication by a nonzero constant) by the operator $\Omega_A^{s_\Phi} \Omega_{A_1}^{p_1 s_\Phi} \cdots \Omega_{A_m}^{p_m s_\Phi}$ when $s_\Phi := d q$. Moreover this compound Cayley operator maps all homogeneous polynomials of $\Pi$ satisfying \eqref{algebra1} and \eqref{algebra3} into the space of invariant polynomials satisfying \eqref{algebra1}, \eqref{algebra2}, and \eqref{algebra3}. By virtue of the calculations above, the space of all such invariant homogeneous polynomials of a fixed degree is spanned by repeated alternating contractions of tensor powers of $\Pi$, where the contractions take place with respect to compatible entries. Specifically this means forming alternating contractions of the multilinear functional
\begin{equation} \Pi^d( \{x_i^1,y_i^1\}_{i=1}^{dq_1},\ldots, \{x_i^m,y_i^m\}_{i=1}^{dq_m} ) := \prod_{j=1}^m \prod_{i=1}^{dq_j} \ang{ x_j^i, \pi_j y_j^i} \label{mytensor2} \end{equation}
in such a way that contractions are in $n$-tuples of indices corresponding to the variables $y_j^i$ for any values of $i$ and $j$ and in $n_j$-tuples of indices corresponding to the variables $x^i_j$ for each $j=1,\ldots,m$. After performing such an operation, the object that remains is a scalar quantity because $d q_j = n_j p_j s_\Phi$ is an integer multiple of $n_j$ and $d (q_1 + \cdots + q_m) = d q (p_1 n_1 + \cdots + p_m n_m) = s_\Phi n$ is an integer multiple of $n$.

Let \[\Lambda := \set{(i,j) \in \Z^2}{ i \in \{1,\ldots,d q_j\}, \ j \in \{1,\ldots,m\}} \]
and suppose $\Lambda$ is given the lexicographic ordering. This is the index set associated to the product \eqref{mytensor2}.
For any $\lambda \in \Lambda$, let its coordinates be denoted $i_\lambda$ and $j_\lambda$, i.e., $\lambda := (i_\lambda,j_\lambda)$. The structure of the expansion of
\begin{equation} \Omega_A^{s_\Phi} \Omega_{A_1}^{p_1 s_\Phi} \cdots \Omega_{A_m}^{p_m s_\Phi} \Pi^d(\{A_1 x_i^1,A y_i^1\}_{i=1}^{dq_1},\ldots,\{A_m x_i^m, A y_i^m \}_{i=1}^{dq_m}) \label{toexpand} \end{equation}
will include a sum over all partitions $J := \{J_1,\ldots,J_{s_\Phi}\}$ of $\Lambda$ into pairwise disjoint sets of cardinality $n$ where alternating contractions of length $n$ are performed over the groups of variables $y_j^i$ indexed by each of the subsets $J_1, \ldots,J_{s_\Phi}$. Summing over all such partitions will yield the expansion of the $\Omega_A^{s_\Phi}$ factor. The expansions of all the remaining factors of $\Omega$ can be expressed as a sum over a different type of partition
$I :=  \{I_1,\ldots,I_{s_\Phi(p_1+\cdots+p_m)}\}$ of $\Lambda$.
In this case, the alternating contractions will involve $n_j$ indices and variables $x_j^{i_1},\ldots,x_j^{i_{n_j}}$ for values of $j$ between $1$ and $m$. In other words, each $I_1,\ldots,I_{s_\Phi(p_1+\cdots + p_m)}$ must consist of indices of the form $\{ (i_1,j), \ldots, (i_{n_j},j) \}$ for some $j$.
While it is perhaps clear what one means by applying the formula \eqref{expandit} to compute the alternating contraction of \eqref{mytensor2} with respect to these partitions $I$ and $J$, carefully carrying out this computation explicitly and compactly requires some additional notation. First, for any $\lambda \in \Lambda$, let $[\lambda]_{I}$ denote the unique subset $I_\ell \in I$ such that $\lambda \in I_\ell$. Likewise let $[\lambda]_{J}$ be the unique element of the partition $J$ containing $\lambda$. Let $\mathfrak S_{I}$ be all permutations of $\Lambda$ such that $[ \sigma_\lambda]_{I} = [\lambda]_{I}$ for all $I$ (i.e., $\mathfrak S_I$ is restricted to permutations of $\Lambda$ which preserve the partition $I$) and analogously for $\mathfrak S_{J}$. 
Lastly, let $r^{I}(\lambda)$ be the total number of indices $\lambda' \in [\lambda]_{I}$ such that $\lambda' \leq \lambda$ and similarly let $c^J(\ell)$ be the total number of indices $\lambda' \in [\lambda]_J$ such that $\lambda' \leq \lambda$. It follows that the repeated alternating contraction of $\Pi^d$ associated to the partitions $I$ and $J$ is given exactly by
\begin{equation} \sum_{\sigma \in {\mathfrak S}_{I}, \tau \in {\mathfrak S}_{J}} (-1)^{\sigma + \tau} \prod_{\lambda \in \Lambda} (\pi_{j_\lambda})_{r^I(\sigma_\lambda) c^J(\tau_\lambda)} \label{allcont1} \end{equation}
where $(\pi_j)_{\ell \ell'}$ is the $\ell \ell'$-entry of the matrix of $\pi_j$ in the standard basis. The formula \eqref{allcont1} can be seen to be an alternating contraction precisely because inside each $I_{\ell} \in I$, $\sigma$ merely permutes elements of $I_{\ell}$, which means that the values of $r^I (\sigma_\lambda)$ for $\lambda \in I_\ell$ are merely permutations of $\{1,\ldots,n_{j_\lambda}\}$ and similarly for the partition $J$. The identity \eqref{expandit} guarantees that \eqref{toexpand} is expressible of a linear combination of terms of the form \eqref{allcont1} with coefficients which depend on the $x_j^i$ and the $y_j^i$; moreover, it can be somewhat easily checked that each term of the form \eqref{allcont1} is invariant under the action of $(A_1,\ldots,A_m,A)$ precisely because \eqref{allcont1} is expressible in terms of alternating contractions and such contractions themselves have the desired invariance properties.

Now for each $\lambda \in \Lambda$, let $\pi_\lambda$ be a $\# \Lambda \times \# \Lambda$ matrix with rows and columns indexed by $\Lambda$ whose entries are
\[ (\pi_\lambda)_{\lambda'\lambda''} := \begin{cases} (\pi_{j_\lambda})_{r^I(\lambda') c^J(\lambda'')} & \mbox{ if } [\lambda]_{I} = [\lambda']_{I} \mbox{ and } [\lambda]_{J} = [\lambda'']_{J} \\
0 & \mbox{ otherwise} \end{cases}.
\]
With this definition, it must be the case that \eqref{allcont1} is equal to
\begin{equation} \sum_{\sigma,\tau \in {\mathfrak S}_{\Lambda}} (-1)^{\sigma + \tau} \prod_{\lambda \in \Lambda} (\pi_\lambda)_{\sigma_\lambda \tau_\lambda} \label{predet1} \end{equation}
where the sums are now over all permutations $\sigma$ and $\tau$ of $\Lambda$ because the terms of the sum  \eqref{predet1} simply vanish for all permutations $\sigma \in {\mathfrak S}_{\Lambda} \setminus {\mathfrak S}_{I}$ and $\tau \in {\mathfrak S}_{\Lambda} \setminus {\mathfrak S}_J$ (simply because there will necessarily be some $\lambda$ such that $[\lambda]_I \neq [\sigma_\lambda]_I$ or $[\lambda]_J \neq [\tau_\lambda]_J$, which means that one of the entries of $\pi_\lambda$ in the product $(\pi_\lambda)_{\sigma_\lambda \tau_\lambda}$ will necessarily be zero by definition of $(\pi_{\lambda})_{\lambda' \lambda''}$
).
The expression \eqref{predet1} is itself exactly equal to the expression
\[ \left( \prod_{\lambda \in \Lambda} \frac{\partial}{\partial t_\lambda} \right) \det \sum_{\lambda \in \Lambda} t_\lambda \pi_\lambda \]
for real parameters $t_\lambda$, since by the product rule
\begin{align*}
 \left( \prod_{\lambda \in \Lambda} \frac{\partial}{\partial t_\lambda} \right) \det \sum_{\lambda \in \Lambda} t_\lambda \pi_\lambda  & =  \left( \prod_{\lambda \in \Lambda} \frac{\partial}{\partial t_\lambda} \right) \sum_{\tau \in {\mathfrak S}_\Lambda} (-1)^{\tau} \prod_{\lambda' \in \Lambda} \left( \sum_{\lambda \in \Lambda} t_\lambda \pi_\lambda \right)_{\lambda' \tau_{\lambda'} } \\
 & = \sum_{\sigma,\tau \in {\mathfrak S}_\Lambda} (-1)^{\tau} \prod_{\lambda' \in \Lambda} (\pi_{\sigma_{\lambda'}})_{\lambda' \tau_{\lambda'}}
 %= \sum_{\sigma, \tau \in {\mathfrak S}_{dm}} (-1)^{\tau} \prod_{\lambda \in \Lambda} (\pi_{\sigma_\lambda})_{\lambda \tau_\lambda} (\tilde \pi_{\sigma_1})_{1 \tau_1} \cdots (\tilde \pi_{\sigma_{dm}})_{dm \tau_{dm}},
\end{align*}
(where the permutation $\sigma$ comes from all orderings of the partial derivatives)
which can be seen to equal \eqref{predet1} by replacing $\tau$ by $\tau \circ \sigma$, reordering the terms in the product, and then replacing $\sigma$ by $\sigma^{-1}$. Derivatives of polynomials can always be evaluated exactly as finite differences, which means that \eqref{allcont1} itself be realized as a linear combination of determinants $\det \sum_{\lambda \in \Lambda} t_\lambda \pi_\lambda$ for various values of the parameters $t_{\lambda}$.

To finish, observe that the matrices $\pi_\lambda$ have common block structure. To be precise, each row $\lambda'$ of the full matrix is uniquely associated with a unique element of $I$, namely, $[ \lambda' ]_{I} \in I$, in the sense that $\pi_{\lambda}$ will be identically zero in row $\lambda'$ unless $[\lambda]_{I} = [\lambda']_I$. The same goes for columns: $\pi_\lambda$ is zero in column $\lambda'$ unless $[\lambda']_{J} = [\lambda]_{J}$. By reordering rows so that rows associated to the same set in $I$ are adjacent and likewise bringing columns associated to the same set in $J$ together to be adjacent, it follows that the alternating contraction \eqref{allcont1} is expressible as a linear combination of determinants of $\#\Lambda \times \#\Lambda = n s_\Phi \times n s_\Phi$ matrices of the exact form described in the statement of the lemma.
To see that every block row associated to $\pi_j$ for fixed $j$ contains no more than $n_j$ nonzero copies of $\pi_j$, simply note that this block row is associated to exactly $n_j$ literal rows $\lambda'$ of the large matrix, and there are exactly $n_j$ values of $\lambda$ such that $\pi_\lambda$ is not automatically zero in this row (namely, the values of $\lambda$ such that $[\lambda]_I = [\lambda']_I$). If each such $\lambda$ belongs to a different element of the column partition $J$, then there can be at most $n_j$ nonzero block entries in this block row. The argument for block entries in block columns is similar.
\end{proof}

 \section{Radon-like operators: Proof of Theorem \ref{genradonthm}}
 \label{proofsec}
 This section contains the proof of Theorem \ref{genradonthm}. The general structure is to combine three elements: the characterization of the Brascamp-Lieb constant given by Lemma \ref{bl2polylem}, the continuous Kakeya-Brascamp-Lieb inequality as it is formulated in Theorem \ref{kakeyathm}, and key ideas from \cite{gressman2018} formulated for the study of nonconcentration inequalities.
The initial step is to observe that the quantity in the integrand on the left-hand side of \eqref{kakeya} is an integral nonconcentration quantity and so may be directly estimated from below via a supremum:
 \begin{lemma}
 Suppose $\pi$ is a continuous map from some $\ell$-dimensional manifold $M$ into $\R^{(n-k) \times n}$. For any Borel set $F \subset M$ and any finite nonnegative Borel measure $\sigma$ on $M$, there is a Borel subset $F' \subset F$ with $\sigma(F') \geq \sigma(F)/2$ such that
 \[ \int_{F^m}  \left[ \blw( \{\pi(t_j)\}_{j=1}^m) \right]^{\frac{1}{p}} d \sigma(t_1) \cdots d \sigma(t_m) \gtrsim (\sigma(F))^m \! \mathop{\sup_{t_1 \in F', \ldots,}}_{t_m \in F'} \!  \left[ \blw( \{\pi(t_j)\}_{j=1}^m) \right]^{\frac{1}{p}} \]
 with an implicit constant which depends only on $n,k$, and $m$. \label{suplemma}
 \end{lemma}
 
 The proof of Lemma \ref{suplemma} is based on the following proposition, which is a mild extension of Lemma 1 from \cite{gressman2018}:
 \begin{proposition}
Let $V$ be a normed vector space. For any positive integer $d$, 
 any topological space $X$, any nonnegative finite Borel measure $\mu$ on $X$, any  $d$-dimensional vector space $\mathcal F$ of continuous functions $f : X \rightarrow V$, and any $\delta \in (0,1)$, there is a closed subset $X_\delta \subset X$ with $\mu( X_\delta) \geq (1 - \delta) \mu(X)$ such that \label{convprop}
\begin{equation}
\mu \left( \set{ x \in X}{ | f(x)| \geq d^{-1} \sup_{y \in X_\delta} |f(y)| } \right) \geq \delta d^{-1} \mu(X) \label{mustbebig}
\end{equation}
for all $f \in {\mathcal F}$. The set $X_\delta$ has the form
\begin{equation} X_{\delta} := \set{x \in X}{ f_j(x) = 0 \ \forall j < j_0 \mbox{ and }  | f_j(x) | \leq 1, \forall j \geq j_0} \label{baseform} \end{equation}
for some functions $f_1,\ldots,f_d \in \mathcal F$ and some $j_0 \in \{0,\ldots,d+1\}$.
\end{proposition}
\begin{proof}
Informally, the content of \eqref{mustbebig} is that there must always be a relatively large subset $X_\delta \subset X$ (large as a fraction of $X$ with respect to the measure $\mu$) such that each $f \in \mathcal F$ exceeds $d^{-1} \sup_{y \in X_\delta} |f(y)|$ on some nontrivial fraction of $X$. In essence, it allows one to approximately reverse the usual inequalities of $L^p$-norms on $X$ if one is allowed to compute the $L^\infty$ norm over a slightly smaller set than all of $X$. The main challenge is to show that the set $X_\delta$ can be defined independently of the particular choice of $f \in \mathcal F$.

By homogeneity of \eqref{mustbebig} and homogeneity of the inequality $\mu(X_\delta) \geq (1-\delta) \mu(X)$ with respect to the measure $\mu$, it may be assumed that $\mu$ is a probability measure since \eqref{mustbebig} is clearly true for the zero measure. For any positive $\epsilon$, let
\[ {\mathcal F}_\epsilon := \set{f \in \mathcal F}{ \mu ( \set{x \in X}{ |f(x)| > 1} ) \leq \epsilon}. \]
The first task is to  establish a number of elementary facts about the sets ${\mathcal F}_{\epsilon}$. The most basic of such facts are that $0 \in {\mathcal F}_\epsilon$ and that $\mathcal F_\epsilon$ is star-shaped at the origin, i.e., $f \in \mathcal F_{\epsilon}$ implies $t f \in \mathcal F_{\epsilon}$ for all $t \in [0,1]$. This follows directly from the inequality $| t f(x)| \leq |f(x)|$ when $t \in (0,1)$. Moreover, for any $f \in {\mathcal F}$, $t f \in {\mathcal F}_\epsilon$ for all sufficiently small $t > 0$, since
\begin{align*}
\lim_{ t \rightarrow 0^+} \mu ( \{|t f| > 1 \}) = \int_X \lim_{t \rightarrow 0^+} \chi_{|t f| > 1} d \mu = 0
\end{align*}
by virtue of Dominated Convergence and the fact that $t f(x) \rightarrow 0$ for all $x$. 
 A fourth important simple fact is that $\mathcal F_\epsilon$ is closed in the vector space topology on $\mathcal F$. To see this, observe that for any sequence of functions $f_n \rightarrow f$ as $n \rightarrow \infty$, at every point $x \in X$ where $|f(x)| > 1$, it will always be the case that $|f_n(x)| > 1$ for all $n$ sufficiently large, simply by continuity of $| \cdot |$. Thus by Dominated Convergence,
\[ \mu ( \{ |f_n| > 1 \} \cap \{ |f| > 1 \} ) \rightarrow \mu ( \{ |f| > 1 \}) \mbox{ as } n \rightarrow \infty. \]
In particular, if $\mu (\{ |f_n| > 1 \}) \leq \epsilon$ for all $n$, then necessarily $\mu( \{|f| > 1\}) \leq \epsilon$.

 Fix a norm $|| \cdot ||_{\mathcal F}$ on $\mathcal F$, and for all $f$ on the unit sphere $\{ ||f||_{\mathcal F} = 1\}$, let
\[ L_{\epsilon} (f) := \sup \set{ t > 0}{ t f \in \mathcal F_{\epsilon}}. \]
This function $L_\epsilon(f)$ is necessarily upper semicontinuous on the unit sphere because $\mathcal F_\epsilon$ is closed: if $L_\epsilon(f) < a$ for some $a > 0$ and some $f$ with $||f||_{\mathcal F} = 1$, then $(a-\eta) f \in \mathcal F_{\epsilon}^c$ for all sufficiently small $\eta > 0$. Because $\mathcal F_\epsilon$ is closed, $(a - \eta) g \in \mathcal F_{\epsilon}^c$ for all $g$ sufficiently close to $f$, yielding $L_\epsilon(g) < a$.  Because the unit sphere is compact, there is a dichotomy: either $L_\epsilon$ is bounded on the unit sphere and $\mathcal F_\epsilon$ is a compact set (since in this case $|| \cdot ||_{\mathcal F}$ is necessarily a bounded function on $\mathcal F_{\epsilon}$), or $L_\epsilon$ is unbounded and there exists a nonzero $f \in \mathcal F$ such that $t f \in \mathcal F_\epsilon$ for all $t > 0$. By Dominated Convergence, any such $f$ must satisfy
\begin{equation} \mu( \{ f \neq 0 \} ) \leq \epsilon \label{itissmall} \end{equation}
because $\lim_{t \rightarrow \infty} \chi_{ |t f(x)| > 1 } = 1$ at every point $x$ where $f(x) \neq 0$.

Now fix any $\delta \in (0,1)$. From here forward, fix $\epsilon := d^{-1} \delta$.
Suppose there exists a nonzero $f_1 \in \mathcal F_\epsilon$ satisfying \eqref{itissmall} when $d = 1$. In this case, setting $X_{\delta} := \set{x \in X}{ f_1(x) = 0 }$ will satisfy the hypotheses of the lemma because all functions $f \in \mathcal F$ will be identically zero on $X_{\delta}$. This forces \eqref{mustbebig} to be vacuously true because the supremum over $X_\delta$ will always be zero. If $d = 1$ and \eqref{itissmall} does not hold for any nonzero $f_1 \in \mathcal F_\epsilon$, one can instead let $f_1 := L_\epsilon(f) f$ for some nonzero $f \in {\mathcal F}_\epsilon$ and define $X_{\delta} := \set{x \in X}{|f_1(x)| \leq 1}$.  Since $f_1 \in \mathcal F_\epsilon$, it must be that $\mu(X_\delta) \geq 1 - \epsilon = (1 - \delta) \mu(X)$. Now 
\[ \mu( \set{x \in X}{ |f_1(x)| \geq 1}) = \lim_{s \rightarrow 1^{-}} \mu(\set{x \in X}{ |f_1(x)| > s}) \]
by Dominated Convergence. If the value of the limit on the right-hand side were strictly less than $\epsilon$, $s^{-1} f_1$ would belong to $\mathcal F_\epsilon$ for some $s < 1$, which would mean that $s^{-1} L_\epsilon f \in {\mathcal F}_\epsilon$, contradicting the maximality of the supremum $L_\epsilon(f)$. Thus
\[ \mu(\set{x \in X}{ |f_1(x)| \geq 1})  \geq \epsilon = d^{-1} \delta \mu(X), \]
which implies \eqref{mustbebig} because $1 \geq \sup_{y \in X_\delta} |f_1(y)|$. By homogeneity of \eqref{mustbebig} in $f$ (and triviality of \eqref{mustbebig} when applied to the zero function), the lemma must hold when $d=1$.

Thus it suffices to assume that  $d > 1$. If $\mathcal F_\epsilon$ is not compact, let $f_1$ be taken to equal any nonzero $f$ satisfying \eqref{itissmall}, let $\tilde X := \set{ x \in X}{ f_1(x) = 0}$, and let $\tilde{\mathcal F}$ be any maximal subspace of $\mathcal F$ which is linearly independent when restricted to $\tilde X$. Because $f_1 = 0$ on $\tilde X$, the dimension $\tilde d$ of $\tilde{\mathcal F}$ is at most $d-1$; if $\tilde{\mathcal F}$ is trivial, then the lemma follows by fixing $X_\delta := \tilde X$. Thus it may be assumed that $1 \leq \tilde d \leq d-1$. By induction on dimension, setting $\tilde{\delta} := \tilde d \delta / (d - \delta) \in (0,1)$ gives that there exists a set ${\tilde X}_{\tilde \delta} \subset \tilde X$ of the form \eqref{baseform} with measure at least $(1-\tilde \delta) (1 - \epsilon) \geq (1-\delta) \mu(X)$ such that
\[ \mu \left( \set{x \in \tilde X}{ |f(x)| \geq \tilde d^{-1} \sup_{y \in {\tilde X}_{\tilde \delta}} |f(y)|} \right) \geq \frac{\tilde \delta (1-\epsilon)}{\tilde d} = \frac{\delta}{d} \mu(X) \]
for all $f \in \tilde{\mathcal F}$; however, every function in $\mathcal F$ restricts to a function in $\tilde{\mathcal F}$ on $\tilde X$, so without loss of generality, the inequality also holds for all $f \in \mathcal F$ with the same constants. Thus \eqref{mustbebig} must be true if one defines $X_{\delta} := \tilde X \cap {\tilde X}_{\tilde \delta}$, which also has the form \eqref{baseform} because $\tilde X$ is merely equal to the set $\set{x \in X}{ f_1(x) = 0}$ for some $f$. 

It now suffices to assume that $\mathcal F_{\epsilon}$ is compact.
Let $\det$ be any nontrivial alternating $d$-linear functional on $\mathcal F$ (which is unique up to scalar multiples). By compactness of $\mathcal F_{ \epsilon}$, there exist $f_{1},\ldots,f_{d} \in {\mathcal F}_{\epsilon}$ such that
\[ | \det (f_{1},\ldots,f_{d}) | = \sup_{ h_1,\ldots,h_{d } \in {\mathcal F}_{\epsilon}}  | \det (h_1,\ldots,h_{d}) |. \]
The supremum must be strictly positive because $|\det(h_1,\ldots,h_d)| \neq 0$ for any linearly independent set $\{h_1,\ldots,h_{d}\} \subset {\mathcal F}_{\epsilon}$ and for any such set, there must exist a small positive constant $t$ such that $t h_i \in {\mathcal F}_{\epsilon}$ for all $i$. Now by Cramer's rule, for any $f^* \in {\mathcal F}_{\epsilon}$,
\begin{equation} f^* = \sum_{j={1}}^{d} (-1)^{j-1} \frac{ \det (f^* , f_{1} ,\ldots, \widehat{f_j}, \ldots, f_d)}{\det (f_{1},\ldots,f_{d})} f_i \label{cramer} \end{equation}
where $\widehat{\cdot}$ denotes omission. By the choice of $f_{1},\ldots,f_{d}$, the coefficients of each $f_i$ in the sum on the right-hand side of \eqref{cramer} has magnitude at most $1$.  If one defines
\[ X_{\delta} := \set{ x \in X}{ |f_j(x)| \leq 1 \ \forall j = 1,\ldots,d}, \]
then $X_\delta^c$ is contained in the union of sets $\set{x \in X}{ |f_j(x)| > 1}$ for $j=1,\ldots,d$; each of these sets has measure at most $\epsilon$, so  $\mu(X_\delta^c) \leq d \epsilon = \delta$.
At any point $x \in X_{\delta}$, each term in the sum \eqref{cramer} has magnitude at most $1$. Thus 
 \begin{equation} \sup_{y \in X_{\delta}} |f(y)| \leq d \label{smallatpoint} \end{equation}
for all $f \in {\mathcal F}_\epsilon$.

Now suppose $f \in {\mathcal F}$ is any function which is not identically zero on $X_{\delta}$ and let $\alpha > 0$ be any number such that
\begin{equation} \alpha <  d^{-1} \sup_{y \in X_{\delta}} |f(y)|, \qquad \text{ i.e.,} \qquad \label{alphabnd}  \sup_{y \in X} |\alpha^{-1} f(y)| > d. \end{equation}
By \eqref{smallatpoint}, $\alpha^{-1} f \in {\mathcal F}$ cannot belong to $\mathcal F_{\epsilon}$. This means that
\[ \mu( \{ |f| > \alpha \}) = \mu ( \{ | \alpha^{-1} f| > 1 \} ) \geq \epsilon = d^{-1} \delta. \]
Taking a supremum over all $\alpha$ satisfying \eqref{alphabnd} and applying Dominated Convergence a final time gives that
\[ \mu \left( \left\{ |f| \geq d^{-1} \sup_{y \in X_{\delta}} |f(y)| \right\} \right) \geq d^{-1} \delta, \]
which is exactly the desired inequality \eqref{mustbebig}.
\end{proof}

\begin{proof}[Proof of Lemma \ref{suplemma}]
By Lemma \ref{bl2polylem}, there is some finite collection $\{\Phi_i\}_{i=1}^N$ of polynomial functions of $\{\pi_j\}_{j=1}^m$ such that
\begin{equation} [ \blw( \{ \pi(t_j) \}_{j=1}^m) ]^{\frac{1}{p}} \approx \sum_{i=1}^N  |\Phi_i( \{\pi(t_j)\}_{j=1}^m)|^{\frac{n-k}{d_i}}, \label{compared} \end{equation}
where $d_i$ is the degree of $\Phi_i$ as in \eqref{algebra1}.  Apply Proposition \ref{convprop} to the vector space $\mathcal F$ of polynomial functions of $\pi$ of degree at most $d_i$, where the measure $\mu$ is $\sigma$ restricted to $F$. It follows, fixing $\delta := 1/2$, that there exists $F'$ with $\sigma(F') \geq \sigma(F)/2$ such that
\begin{align*}
\int_{F}  & | \Phi_i(  \{\pi(t_j)\}_{j=1}^m)  |^{\frac{n-k}{d_i}} d \sigma(t_1)  \\
& \geq \int_{F} | \Phi_i(   \{\pi(t_j)\}_{j=1}^m)  |^{\frac{n-k}{d_i}}  \chi_{|\Phi_i(\{\pi(t_j)\}_{j=1}^m)| \geq  \frac{\sup_{t_1 \in F'} |\Phi_i(\{\pi(t_j)\}_{j=1}^m)|}{\dim \mathcal F}} d \sigma(t_1) \\
& \geq \left(\frac{\sup_{t_1 \in F'} |\Phi_i(\{\pi(t_j)\}_{j=1}^m)|}{\dim \mathcal F} \right)^{\frac{n-k}{d_i}} \\ & \qquad \cdot \sigma \left( \set{t_1 \in F}{|\Phi_i(\{\pi(t_j)\}_{j=1}^m)| \geq  \frac{\sup_{t_1 \in F'} |\Phi_i(\{\pi(t_j)\}_{j=1}^m)|}{\dim \mathcal F}} \right) \\
& \geq \frac{1}{2} (\dim \mathcal F)^{-1 - \frac{n-k}{d_i}} \sigma(F) \left( |\Phi_i(\{\pi(t_j)\}_{j=1}^m)|\right)^{\frac{n-k}{d_i}} \chi_{F'}(t_1)
\end{align*}
for any values of $t_1,t_2,\ldots,t_m$. Note the slight abuse of notation in the inequality just derived: on the top line (which becomes the left-hand side), $t_1$ denotes a variable of integration, while on the final line (the new right-hand side), $t_1$ denotes a point which can be chosen arbitrarily (but yields a trivial inequality unless $t_1 \in F'$).
We proceed inductively, integrating this inequality over $t_2$ and deriving a new inequality, etc.; the final result of this process yields the inequality
\begin{align*} \int_{F^m} & | \Phi_i(  \{\pi(t_j)\}_{j=1}^m)  |^{\frac{n-k}{d_i}} d\sigma(t_1) \cdots d \sigma(t_m) \\ & \gtrsim \left( |\Phi_i(\{\pi(t_j)\}_{j=1}^m)|\right)^{\frac{n-k}{d_i}} (\sigma(F))^m \prod_{j=1}^{m} \chi_{F'}(t_j), \end{align*}
where the implicit constant is a function of $\dim \mathcal F$. Summing over $i$ and taking a supremum of the right-hand side over all $t_1,\ldots,t_m \in F'$ completes the lemma by virtue of \eqref{compared}.
\end{proof}

With the proof of Lemma \ref{suplemma} in hand, the proof of Theorem \ref{genradonthm} follows rather easily as well:
\begin{proof}[Proof of Theorem \ref{genradonthm}]
Suppose that $\inc \subset \Omega \subset \R^n \times \R^n$ is a left-algebraic incidence relation with defining function $\rho : \Omega \rightarrow \R^{n-k}$. By Theorem \ref{kakeyathm}, for any Borel measurable set $E \subset \R^n$, the function
\[ \widetilde T_m \chi_E(x)  := \int_{\li{x} \cap E} \cdots \int_{\li{x} \cap E} [ \blw(\{D_x \rho(x,y_j)\}_{j=1}^m )]^{\frac{1}{p}} d \sigma(y_1) \cdots d \sigma(y_m) \]
belongs to $L^p(\R^n)$ with $p := n / (m(n-k))$ and satisfies
\[ || T_m \chi_E ||_{L^p(\R^n)} \lesssim |E|^{m} \]
with implicit constant which is independent of $E$. Now apply Lemma \ref{suplemma} by fixing $F$ to be any subset of $\li{x} \cap E$ on which $\sigma$ is finite; this gives that
\[ \widetilde T_m \chi_E(x) \gtrsim (\sigma (F))^m \sup_{y_1,\ldots,y_m \in F'} [ \blw(\{D_x \rho(x,y_j)\}_{j=1}^m) ] ^\frac{1}{p} \]
for some Borel set $F' \subset F \subset E \cap \li{x}$ with $\sigma(F') \geq \sigma(F)/2$ and some implicit constant which is independent of $E$ and $x$. The main hypothesis of Theorem \ref{genradonthm} gives that
\begin{equation*} \sup_{y_1,\ldots,y_m \in F'} [ \blw(\{D_x \rho(x,y_j)\}_{j=1}^m) ] ^\frac{1}{p} \gtrsim (\sigma(F'))^{s} \gtrsim (\sigma(F))^s  \end{equation*}
for some exponent $s$ and an implicit constant independent of $x$ and $F'$ and consequently independent of $E$. But $\sigma(E \cap \li{x}) = T \chi_E(x)$ for the Radon-like operator \eqref{theradondef}, and also $\sigma$ is $\sigma$-finite on the manifold $\li{x}$ since it has smooth density with respect to Lebesgue measure there,
so by applying the newly-derived inequality $\tilde T_m \chi_E(x) \gtrsim (\sigma(F))^{m+s}$ to a sequence of choices of $F$ selected so that $\sigma(F) \rightarrow \sigma(E \cap \li{x})$ in the limit, it follows that
\[ \tilde T_m \chi_E(x) \gtrsim (T \chi_E(x))^{m+s} \]
with implicit constant that is independent of $x$ and $E$. It follows that
\[ || (T \chi_E)^{m+s} ||_{L^p(\R^n)} \lesssim || \widetilde T_m \chi_E ||_{L^p(\R^n)} \lesssim |E|^m. \]
Raising both sides to the power $1/(m+s)$ gives \eqref{theradonineq}.
\end{proof}

\section{Applications of Theorem \ref{genradonthm}}
\label{appsec}
This final main section looks at various applications of Theorem \ref{genradonthm}, which includes the proof of Theorem \ref{radonthm}. It begins with some basic computations which show how to compute a suitable defining function and the measures \eqref{themeasure} for a Radon-like operator whose incidence relation is given parametrically. Following that is an example application of Theorem \ref{genradonthm} which yields an alternative to Christ's proof of the $L^p$-improving properties of the moment curve \cite{christ1998}. Then comes the proof of Theorem \ref{radonthm}, followed by a few extensions and generalizations.

\subsection{A preliminary observation about parametrized incidence relations}
\label{radonsec}
%In this section, we compute the measure $d \sigma$ defined by \eqref{themeasure} when the submanifolds $\li{x}$ are given via parametrization rather than via defining function. 
\begin{proposition}
Let $x,y \in \R^n$ be regarded as ordered pairs $(x',x''), (y',y'') \in \R^{k} \times \R^{n-k}$ and let $\gamma : \R^k \times \R^n \rightarrow \R^{n-k}$ be any polynomial function. Then the Radon-like operator given by \label{pmeasure}
\[ T f(x) := \int_{\R^k} f( x' + t, x'' + \gamma(t,x)) dt \]
is exactly the operator \eqref{theradondef} from Theorem \ref{genradonthm} for the defining function 
\begin{equation} \rho(x,y) = y'' - x'' - \gamma (y' - x', x). \label{deffunc} \end{equation}
In particular, the measure $d \sigma$ defined by \eqref{themeasure} equals Lebesgue measure $dt$.
\end{proposition}
\begin{proof}
Let $B(t,x)$ be the $(n-k) \times k$ matrix given by
\[  \left[ \begin{array}{ccc} 
\frac{\partial \gamma_1}{\partial t_1}(t,x) & \cdots &  \frac{\partial \gamma_1}{\partial t_k} (t,x)  \\
\vdots & \vdots & \vdots \\
\frac{\partial \gamma_{n-k}}{\partial t_1}(t,x) & \cdots &  \frac{\partial \gamma_{n-k}}{\partial t_k} (t,x)
\end{array} \right], \]
where $\gamma_1,\ldots,\gamma_{n-k}$ are the coordinate functions of $\gamma$ in the standard basis and $t_1,\ldots,t_k$ are the coordinates of $t$. Taking \eqref{deffunc} as the definition of $\rho$, the right derivative matrix $D_y \rho$ (recall \eqref{incmat}) has the block structure 
\[ \left[ - B(y'-x',x) \ \ I_{n-k}  \right] \]
where $I_{n-k}$ is the $(n-k) \times (n-k)$ identity.
The induced Riemannian metric $\ang{\cdot,\cdot}$ on the graph $\li{x}$ satisfies
\begin{equation} \ang{\frac{\partial}{\partial t_i}, \frac{\partial}{\partial t_j}} = \delta_{i,j} + \frac{\partial \gamma}{\partial t_i} \cdot \frac{\partial \gamma}{\partial t_j}, \label{riemann} \end{equation}
where $\cdot$ is the usual dot product on $\R^{n-k}$ and $\delta_{i,j}$ is the Kronecker delta.
When the right-hand side of \eqref{riemann} is regarded as a matrix, the square root of the determinant equals the density of Hausdorff measure with respect to coordinate measure, i.e.,
\[ d \mathcal H^k = \det (I_{k} + B^T B)^{1/2} dt. \]
Similarly, 
\[ \det (D_y \rho (D_y \rho)^T)^{1/2} =  \det (I_{n-k} + B B^T)^{1/2}. \]
Therefore
\[ \frac{d \mathcal H^k}{\det (D_y \rho (D_y \rho)^T)^{1/2}} = \frac{ \det (I_{k} + B^T B)^{1/2}}{ \det (I_{n-k} + B B^T)^{1/2}} dt. \]
Now both $\det (I_{k} + B^T B)^{1/2}$ and $\det (I_{n-k} + B B^T)^{1/2}$ are invariant under the transformation $B \mapsto O_{n-k} B O_k$ where $O_{n-k}$ and $O_{k}$ are orthogonal matrices of size $(n-k) \times (n-k)$ and $k \times k$, respectively. Thus by the Singular Value Decomposition, to compute the ratio
\[ \frac{ \det (I_{k} + B^T B)^{1/2}}{ \det (I_{n-k} + B B^T)^{1/2}}, \] it suffices to assume that the only nonzero entries of $B$ appear on the diagonal and that $B_{ii} \geq 0$ for all $i$, in which case
\[ \det (I_{k} + B^T B)^{1/2} = \det (I_{n-k} + B B^T)^{1/2} = \prod_{i=1}^{\min\{k,n-k\}} (1 + B_{ii}^2)^{1/2}. \]
It follows that $d \sigma = d t$.
\end{proof}

\subsection{Warm-up application: The moment curve}

As a first example of how Theorem \ref{genradonthm} can be applied in practice, consider the case of convolution with the standard measure on the so-called moment curve. In $\R^n$ this is exactly the Radon-like transform given by
\begin{equation} T f(x) := \int f(x_1 + t, x_2 + t^2,\ldots, x_n + t^n) dt. \label{moment} \end{equation}
This operator was the titular case study of Christ's seminar work on the combinatorial approach to $L^p$-improving inequalities \cite{christ1998}. In particular, Christ established that this operator satisfies a restricted weak type $(\frac{n+1}{2},\frac{n(n+1)}{2(n-1)})$ and a corresponding dual inequality. Christ's method was later extended by Stovall to arrive at a full Lebesgue space bound for this and more general polynomial curves \cites{stovall2009,stovall2010}.  The arguments below show that Theorem \ref{genradonthm} provides a rather direct route to an intermediate result, namely that \eqref{moment} satisfies a restricted strong type $(\frac{n+1}{2},\frac{n(n+1)}{2(n-1)})$ inequality.

As implied above, let $x := (x_1,\ldots,x_n)$ and $y := (y_1,\ldots,y_n)$.
The incidence relation associated to \eqref{moment} has an algebraic defining function which is given by
\[ \rho(x,y) := (x_2 - y_2 + (y_1 - x_1)^2,\ldots,x_n - y_n + (y_1 - x_1)^n). \]
Proposition \ref{pmeasure} guarantees that the operator \eqref{moment} equals the operator \eqref{theradondef} specified by Theorem \ref{genradonthm}.
A simple computation gives that $D_x \rho(x,y) = \pi(y_1 - x_1)$, where
\[ \pi(t) := \left[ \begin{array}{ccccc}  2t & 1 & 0 & \cdots & 0 \\
- 3 t^2 & 0 & 1 & \ddots & \vdots \\
\vdots & \vdots & \ddots & \ddots & 0 \\
(-1)^{n} n t^{n-1} & 0 & \cdots & 0 & 1
\end{array} \right] \]
There is a centrally-important polynomial function $\Phi(t^{(1)},\ldots,t^{(n)})$ which depends only on $\pi(t^{(1)}),\ldots,\pi(t^{(n)})$ and satisfies the invariance properties \eqref{algebra1} and \eqref{algebra2}, given (as in Lemma \ref{itsadet}) by a block-form determinant:
\[ \Phi(t^{(1)},\ldots,t^{(n)}) := \det \left[ \begin{array}{cccc}
\pi (t^{(1)}) & 0 & \cdots & 0 \\
0 & \pi (t^{(2)}) & \ddots & \vdots \\
\vdots & \ddots & \ddots &  \vdots \\
0 & \cdots & 0 & \pi(t^{(n-1)})  \\
\pi (t^{(n)}) & \pi(t^{(n)}) & \cdots & \pi(t^{(n)})
\end{array} \right].  \]
Subtracting upper block rows from the bottom block row results in individual block entries which are zero in all but their first columns. Expanding the determinant in the columns which vanish in the last block row gives that $\Phi$ must equal $\pm (n!)$ times
\[ \det \left[ \begin{array}{ccc}  t^{(1)} - t^{(n)} & \cdots &  t^{(n-1)} - t^{(n)} \\  \vdots & \ddots & \vdots \\   ( t^{(1)})^{n-1} - (t^{(n)})^{n-1} & \cdots &  ( t^{(1)})^{n-1} - (t^{(n)})^{n-1}  \end{array} \right], \]
which is equal to 
\[ (-1)^n \det \left[ \begin{array}{ccc} 1 & \cdots & 1 \\ t^{(1)} & \cdots & t^{(n)} \\ \vdots & \ddots & \vdots \\ (t^{(1)})^{n-1} & \cdots & (t^{(n)})^{n-1} \end{array} \right]. \]
This is simply the classical Vandermonde determinant. Now if $F \subset \R$ is any Borel measurable set with positive Lebesgue measure, it is always possible to find $n$ distinct points $t^{(1)},\ldots,t^{(n)} \in F$ such that $|t^{(i)} - t^{(j)}| \geq |F|/(2n-1)$ whenever $i \neq j$. This is because one can always partition $\R$ into nonoverlapping intervals of length $|F|/(2n-1)$; the set $F$ must intersect at least $(2n-1)$ of these intervals in a set of positive measure, so one can always take $t^{(1)},\ldots,t^{(n)}$ from $n$ such intervals which are not adjacent.  Thus
\[ \sup_{t^{(1)},\ldots,t^{(n)}  \in F}  \! \! \! \frac{|\Phi(t^{(1)},\ldots,t^{(n)})|}{n!} = \sup_{t^{(1)},\ldots,t^{(n)}  \in F} \prod_{1 \leq i < j \leq n}  \! \! |t^{(i)} - t^{(j)}| \geq \frac{|F|^{\frac{n(n-1)}{2}} }{(2n-1)^{\frac{n(n-1)}{2}}} . \]
Since $\Phi$ is a degree $n-1$ function of each $\pi(t^{(j)})$ in the sense of \eqref{algebra1}, the inequality \eqref{blblock} gives that
\[ \sup_{t^{(1)},\ldots,t^{(n)}  \in F} \left[ \blw( \{ \pi(t^{(j)}) \}_{j=1}^n) \right]^{\frac{1}{p}} \gtrsim |F|^{\frac{n(n-1)}{2}}. \]
Thus Theorem \ref{genradonthm} implies that \eqref{moment} satisfies a restricted strong type $(\frac{n+1}{2},\frac{n(n+1)}{2(n-1)})$ inequality.

\subsection{Results concerning nonconcentration inequalities}

Before proceeding with the proof of Theorem \ref{radonthm}, it is necessary to recall the main result from \cite{gressman2018} concerning nonconcentration inequalities.  The point of doing so is to give sufficient conditions of a quantitative nature which guarantee that the main hypothesis \eqref{nonconcentrate} of Theorem \ref{genradonthm} is true. This will involve identifying certain invariant quantities which generalize the notion of rotational curvature, first introduced by Phong and Stein \cite{ps1989}.

From \cite{gressman2018}, recall that a multisystem $\mso$ of size $N$ on an open set $\Omega \subset \R^{n-k}$ is a collection of smooth vector fields $\{X_j^{i}\}_{j=1,\ldots,n-k, \ i = 1,\ldots,N}$ such that for each $i=1,\ldots,N$, the vector fields $\{ X_j^i\}_{j=1,\ldots,n-k}$ commute and are linearly independent at every point in $\Omega$. The collection of all such multisystems is denoted $\MS^{(N)}$. For any fixed vectors $X_1,\ldots,X_{n-k}$ at a point $t \in \Omega$ and any function $\alpha : \{1,\ldots, \ell\} \rightarrow \{1,\ldots,n-k\}$, where $\ell \leq N$, the differential operator $(X \ms)^{\alpha}$ is defined to equal $Z^\ell_{\alpha_\ell} \cdots Z^1_{\alpha_1}$, where $Z^i_j$ is the unique constant-coefficient linear combination of $X^i_1,\ldots,X^i_{n-k}$ which equals $X_j$ at the point $t$.
Such $\alpha$ will be called ordered multiindices in $n$ variables and $|\alpha|$ will be used to denote the order of differentiation of $(X \ms)^\alpha$, i.e., $|\alpha| = \ell$. Matrices $T \in \GL{n-k}$ act on these differential operators by defining
\[ (T^* X)_i := \sum_{j=1}^{n-k} T_{ji} X_j \]
and taking $(T^* X \ms)^{\alpha} := ((T^* X) \ms)^{\alpha}$.  The main result from \cite{gressman2018} that will be used here is the following:
\begin{theorem}[cf. Theorem 4 of \cite{gressman2018}]
Suppose $\Omega \subset \R^{n-k}$ is an open set and that $\Phi(t_1,\ldots,t_m)$ is a polynomial function of $t_1,\ldots,t_m \in \R^{n-k}$.
For any $s > 0$, let  \label{betterthm}
\begin{equation}
\omega(t)   :=   \mathop{\inf_{\mso \in \MS^{(N)}}}_{T \in \GL{n-k}} \max_{|\alpha_1|,\ldots,|\alpha_m| \leq N} \frac{ \left| (T^* e \ms)^{\alpha_1}_1 \cdots (T^* e \ms)^{\alpha_m}_m \Phi(t,\ldots,t) \right|^{\frac{1}{s}}}{|\det T|} 
 \label{betterdens}
\end{equation}
where $e := \{e_j\}_{j=1}^{n-k}$ is the collection of standard coordinate vectors at $t$ and $(T^* e \ms)^{\alpha_j}_j$ denotes the differential operator $(T^* e \ms)^{\alpha_j}$ applied in the variable $t_j$. If $\sigma$ is any nonnegative Borel measure which is absolutely continuous with respect to Lebesgue measure such that
\[ \frac{d \sigma}{dt}(t) \leq \omega(t) \]
at each point $t \in \Omega$, where $\frac{d\sigma}{dt}$ is the Radon-Nikodym derivative of $\sigma$ with respect to Lebesgue measure, then for any Borel set $F \subset \Omega$,
\begin{equation} \sup_{t_1,\ldots,t_m \in F} |\Phi(t_1,\ldots,t_m)| \gtrsim \left[ \sigma(F) \right]^s \label{betterineq} \end{equation}
with implicit constant depending only on $(n-k,m,s,\deg \Phi,N)$.
\end{theorem}

Suppose $\Phi(t_1,\ldots,t_m)$ is a polynomial function of $t_1,\ldots,t_m \in {\mathbb R}^{n-k}$ and that $c_1,\ldots,c_m$ are nonnegative integers such that
\begin{equation} \partial^{\alpha_1}_{t_1} \cdots \partial^{\alpha_m}_{t_m} \Phi(t_1,\ldots,t_m) \equiv 0 \label{vcond} \end{equation}
identically on the diagonal $t_1 = \cdots = t_m$ for all choices of $\alpha_1,\ldots,\alpha_m$ satisfying $|\alpha_j| \leq c_j$ for each $j$ and $|\alpha_j| < c_j$ for at least one $j=1,\ldots,m$. By definition of $(T^* e \ms)^{\alpha}$,
\[ (T^* e \ms)^{\alpha} = Z^{\ell}_{\alpha_\ell} \cdots Z_{\alpha_1}^{1} \]
where $Z^1_{\alpha_1}$ is a linear combination of $X^1_1,\ldots,X^1_{n-k}$ which equals $\sum_{j=1}^{n-k} T_{j \alpha_1} \partial_j$ at the point $t$ and so on through $Z^{\ell}_{\alpha_\ell}$, which is a linear combination of $X^\ell_1,\ldots,X^\ell_{n-k}$ that equals $\sum_{j=1}^{n-k} T_{j \alpha_\ell} \partial_j$ at the base point $t$. For convenience, let $T^* \partial$ denote the tuple
\[ \left( \sum_{j=1}^{n-k} T_{j1} \partial_j,\ldots, \sum_{j=1}^{n-k} T_{j(n-k)} \partial_j \right) \]
and let $(T^* \partial)^{\alpha}$ be the composition 
\[  \left( \sum_{j=1}^{n-k} T_{j \alpha_1} \partial_j \right) \cdots \left( \sum_{j=1}^{n-k} T_{j \alpha_\ell} \partial_j \right). \] The difference $(T^* e \ms)^{\alpha} - (T^* \partial)^{\alpha}$ is a differential operator of order strictly less than $\ell$ at that distinguished point $t$ where each $Z^\ell_i$ is fixed to equal $\sum_{j=1}^{n-k} T_{ji} \partial_j$. By hypothesis on the vanishing of derivatives of $\Phi$ on the diagonal, then, it follows that
\[ (T^* e \ms)^{\alpha_1}_1 \cdots (T^* e \ms)^{\alpha_m}_m \Phi(t,\ldots,t) = (T^* \partial)_1^{\alpha_1} \cdots (T^* \partial)_m^{\alpha_m} \Phi(t,\ldots,t) \]
when $|\alpha_j| = c_j$ for each $j=1,\ldots,m$. As before, the subscript $j$ in the expression $(T^* \partial)_j$ refers to the partial derivative as it is applied in the variable $t_j$. It follows that
\begin{equation} \omega(t) \geq \inf_{T \in \GL{n-k}} \max_{|\alpha_1| = c_1,\ldots, |\alpha_m| = c_m} \frac{ \left| (T^* \partial)_1^{\alpha_1} \cdots (T^* \partial)_m^{\alpha_m} \Phi(t,\ldots,t) \right|^{\frac{1}{s}}}{|\det T|}. \label{simpler1} \end{equation}

To further aid in the estimation of the right-hand side of \eqref{simpler1}, one may assume without loss of generality that the infimum over $T$ is taken only over those $T$ which are upper-triangular. The reason for this is that we may always write $T E = U$ for some matrix $E$ of determinant $1$ with uniformly bounded entries (i.e., a bound independent of $T$) and some upper-triangluar matrix $U$, which then implies that
\begin{equation} \begin{aligned} & \left| (U^* \partial)_1^{\alpha_1}   \cdots (U^* \partial)_m^{\alpha_m} \Phi(t,\ldots,t) \right| \\ & \qquad \lesssim \max_{|\beta_1| = |\alpha_1|,\ldots,|\beta_m| = |\alpha_m|} \left|  (T^* \partial)_1^{\alpha_1} \cdots (T^* \partial)_m^{\alpha_m} \Phi(t,\ldots,t) \right| \end{aligned} \label{simpler2} \end{equation}
with universal implicit constants depending only on $n$. The proof of this fact is a direct application of the following proposition:
\begin{proposition}
For every positive integer $d$ and every $T \in \R^{d \times d}$, there exist $U,E \in \R^{d \times d}$ such that $U = T E$, $U$ is upper-triangular, $\det E = 1$, and
\[ \sum_{\ell=1}^d \sum_{i=1}^d |E_{\ell i}| \leq 2^{d} - 1. \]
\end{proposition}
\begin{proof}
If $d=1$, the proposition is trivially true simply by fixing $E$ to be the $1 \times 1$ identity matrix. When $d > 1$, 
suppose that the final row of $T$ has at least one nonzero entry. Let $i$ be an index which maximizes $|T_{di}|$. Without loss of generality, it may be assumed that $i = d$, since otherwise we may permute columns of $T$ to make it so, and compensate with a corresponding permutation of the rows of the matrix $E$ to be constructed shortly (and if the permutation leaves $\det E$ negative, simply multiply a single column of $E$ by $-1$ to restore positivity).  Under this assumption, let $T'_{ji} = T_{ji} - T_{di} T_{dd}^{-1} T_{jd}$ for all $i,j \in \{1,\ldots,d-1\}$. By induction, there exists $E' \in \R^{(d-1) \times (d-1)}$ with determinant $1$ such that $T' E'$ is upper triangular. Now let $E$ be defined so that
\[ 
E_{\ell i} := \begin{cases} E'_{\ell i} & \ell,i \in \{1,\ldots,d-1\} \\
- \sum_{j=1}^{d-1} T_{dj} T_{dd}^{-1} E_{j i}' & \ell = d \text{ and }  i \in \{1,\ldots,d-1\} \\
0 & i = d \text{ and } \ell \in \{1,\ldots,d-1\} \\
1 & i = \ell = d
\end{cases}.
\]
 Then for $i \in \{1,\ldots,d-1\}$, we have
\[ \sum_{\ell=1}^d T_{j \ell} E_{\ell i} = \left( \sum_{\ell =1}^{d-1} T_{j \ell} E'_{\ell i} \right) - T_{j d} \sum_{\ell=1}^{d-1} \frac{T_{d \ell}}{T_{dd}}  E'_{\ell i} = \begin{cases}  \sum_{\ell=1}^{d-1} T'_{j\ell} E'_{\ell i} & j \neq d \\ 0 & j = d \end{cases}, \]
which ensures that $T E$ is indeed an upper-triangular matrix. We have that $\det E = \det E'$ by expanding the determinant of $E$ with respect to its $d$-th column. Lastly, if $i < d$, we have
\[ \sum_{\ell=1}^d |E_{\ell i}| = \sum_{\ell=1}^{d-1} |E'_{\ell i}| + \left| \sum_{\ell=1}^{d-1} \frac{T_{d \ell}}{T_{dd}} E'_{\ell i} \right| \leq 2 \sum_{\ell=1}^{d-1} |E'_{\ell i}|\]
because $|T_{d \ell}/T_{dd}| \leq 1$ for each $\ell \in \{1,\ldots,d-1\}$. Thus
\[ \sum_{i,\ell=1}^d |E_{\ell i}| \leq 2 \sum_{i,\ell=1}^{d-1} |E'_{\ell i}| + \sum_{\ell=1}^{d} |E_{\ell d}| \leq 2(2^{d-1}-1) + 1 = 2^{d}-1. \]

If $T_{di} = 0$ for all $i$, one can instead take $E$ just as above with the exception that $E_{di} := 0$ for $i \in \{1,\ldots,d-1\}$. The desired conclusion follows after a minor modification of the above argument.

As a final remark, it may be of interest to note that a modification of this argument which involves a further step of multiplying both $U$ and $E$ on the right by a suitably optimized diagonal matrix yields the stronger inequality $\max_{i} \sum_\ell |E_{\ell i}| \leq 2^{(d-1)/2}$. The extent to which this upper bound can be improved as a function of $d$ is not immediately clear, but this will not be a concern under the present circumstances.
\end{proof}

\subsection{Quadratic submanifolds: Proof of Theorem \ref{radonthm}}
\label{radonthmproof}

Just as was done for the moment curve, the main idea behind the proof of Theorem \ref{radonthm} is to apply Theorem \ref{genradonthm}; to do so, one establishes the nonconcentration inequality \eqref{nonconcentrate} by studying a well-chosen invariant polynomial $\Phi$ and applying Lemma \ref{bl2polylem}.

To be more specific, the proof proceeds by applying Theorem \ref{genradonthm} to the operator \eqref{myoperator}.  The paramter $s$ in Theorem \ref{genradonthm} will be fixed to equal $n-k$ and $m$ will be taken equal to $n$. For an appropriate defining function $\rho$, the problem reduces to proving that
\[ \sup_{y_1,\ldots,y_m \in F} |\blw(\{D_x \rho(x,y_j)\}_{j=1}^n)|^{\frac{1}{p}} \gtrsim (\sigma(F \cap \li{x}))^{n-k} \]
uniformly for all $x \in \R^n$ and all Borel $F \subset \li{x}$. To accomplish this, it suffices to identify a suitable invariant polynomial function $\Phi$ of the matrices $\{D_x \rho(x,y_j)\}_{j=1}^n$ which satisfies the inequality
\[ |\blw(\{D_x \rho(x,y_j)\}_{j=1}^n)|^{\frac{1}{p}} \gtrsim |\Phi(\{D_x \rho(x,y_j)\}_{j=1}^n)| \]
uniformly in $x$ and $y_1,\ldots,y_n$. Since $D_x \rho(x,y)$ will depend only on the first $k$ coordinates of $y$ and since Proposition \ref{pmeasure} guarantees that $\sigma$ agrees with Lebesgue measure in these first $k$-coordinates, it will suffice by Theorem \ref{betterthm} and the inequalities \eqref{simpler1} and \eqref{simpler2} to show that (with $c := n-k$ here and throughout the rest of the section) 
\begin{equation} \begin{aligned} \max_{|\alpha_1| = \cdots = |\alpha_k| = c} & | (U^* \partial)^{\alpha_1}_1 \cdots (U^* \partial)^{\alpha_k}_k \Phi (D_x \rho(x,y) ,\ldots,D_x \rho(x,y))| \\ & \geq |\det U|^{c} \left| \prod_{j=0}^{k-1} \det \left[ \begin{array}{ccc} \lambda_{1 (jc+1)} & \cdots & \lambda_{1(j c + c) } \\ \vdots & \ddots & \vdots \\ \lambda_{c(j c + 1) } & \cdots & \lambda_{c(jc + c)} \end{array} \right] \right|  \end{aligned} \label{goal}
\end{equation}
for any upper-triangular matrix $U \in \R^{k \times k}$, where as before, the operator $(U^* \partial)_j^{\alpha_j}$ is applied with respect to the variables of $y_j$ prior to restricting to the diagonal.

To arrive at the final goal \eqref{goal}, one must first be precise about the defining function $\rho$ and the polynomial $\Phi$ to be used. As for $\rho$, it is convenient to use \eqref{deffunc} multiplied by a factor of $-1$ to simplify computation:
\[ \rho_j(u,v) := -v_{k+j} + u_{k+j} + \frac{1}{2} \sum_{i=1}^{k} \lambda_{ji} (v_i-u_i)^2 \]
for $j=1,\ldots,c$. Here $u := (u_1,\ldots,u_n) \in \R^{n}$ and $v := (v_1,\ldots,v_n) \in \R^n$ (where the symbols $u$ and $v$ are used to simply avoid the need to temporarily redefine the meaning of the subscripted variables $y_1,\ldots,y_n$). Taking the unusual but harmless convention of ordering the entries of $u$ as $u_{k+1},\ldots,u_n,u_1,\ldots,u_k$, the
The corresponding left derivative matrix of $\rho$ is given by
\begin{equation} D_u \rho = \left[ \begin{array}{ccccccc} 1 & 0 & \cdots & 0 & (u_1 - v_1) \lambda_{11} & \cdots & (u_k - v_k) \lambda_{1k} \\
0 & 1 & \ddots & \vdots & \vdots & \ddots & \vdots \\
\vdots & \ddots & \ddots & 0 & \vdots & \ddots & \vdots \\
0 & \cdots & 0 &  1 & (u_1 - v_1) \lambda_{c1} & \cdots & (u_k - v_k) \lambda_{ck}\end{array} \right]. \label{theld} \end{equation}
As already noted, the case $m := n$ of Theorem \ref{genradonthm} is the one of interest here, and the quantity $W ( \{D_x \rho(x,y_j)\}_{j=1}^n$ will be estimated from below in terms of well-chosen invariant polynomials (where once again it should be emphasized that each $y_j$ is still to be understood as an element of $\R^n$ for each $j=1,\ldots,n$ as opposed simply a coordinate entry of some single vector). In particular, by Lemma \ref{bl2mv} and specifically using \eqref{blblock2}, it will be the case that
\begin{equation} \left[ W ( \{D_x \rho(x,y_j)\}_{j=1}^n \right]^{\frac{1}{p}} \gtrsim | \Phi \{D_x \rho(x,y_j)\}_{j=1}^n| \label{mainlb} \end{equation}
whenever $\Phi$ satisfies \eqref{algebra2} and \eqref{algebra1} with $d_1 = \cdots = d_n = d-k$, as shall be the case for the specific $\Phi$ constructed below.

As in earlier sections, suppose that $\pi_1,\ldots,\pi_n$ are real $c \times n$ matrices. To these matrices one may associate an $nc \times nc$ matrix $M(\pi_1,\ldots,\pi_n)$ as follows. First, regard each $\pi_j$ as posessing $c \times c$ block $A_j$ and a $c \times k$ block $B_j$ by fixing $A_j$ to consist of the first $c$ columns of $\pi_j$ and to consist of $B_j$ the final $k$ columns of $\pi_j$. The matrix $M$ will have a nested block structure: 
\begin{itemize}
\item an upper left block $M^{UL}$ of size $kc \times c^2$ which itself is divided into smaller blocks of size $c \times c$ which are denoted $M^{UL}_{ij}$ for $i=1,\ldots,k$, $j=1,\ldots,c$,
\item a lower left block $M^{LL}$ of size $c^2 \times c^2$ which is itself divided into smaller $c \times c$ blocks $M^{UR}_{ij}$ for $i,j=1,\ldots,c$,
\item an upper right block $M^{UR}$ of size $kc \times kc$ consisting of smaller $c \times k$ blocks $M^{UR}_{ij}$ for $i=1,\ldots,k$, $j=1,\ldots,c$, and 
\item a lower right block $M^{LR}$ of size $c^2 \times kc$ consisting of smaller $c \times k$ blocks $M^{LR}_{ij}$ for $i,j=1,\ldots,c$. 
\end{itemize}
The various sub-blocks of $M$ are derived from the matrices $A_j$ and $B_j$ as follows:
\begin{itemize}
\item Let $M^{UR}_{ij} = B_i$ if the diagonal of $M^{UR}$ passes through $M^{UR}_{ij}$ and let $M^{UR}_{ij} = 0$ otherwise. (Here the diagonal is understood as the literal diagonal of the $kc \times kc$ matrix $M^{UR}$.)
\item Let $M^{UL}_{ij} = A_i$ if $(i,j)$ is a pair for which $M^{UR}_{ij}$ lies on the diagonal of $M^{UR}$ and $M^{UL} = 0$ otherwise.  The layout of $M^{UL}$ matches the layout of $M^{UR}$ with the $B_i$ blocks replaced by $A_i$ blocks.
\item Let $M^{LL}_{ii} = A_{k+i}$ and $M^{LL}_{ij} = 0$ when $i \neq j$.
\item Let $M^{LR}_{ii} = B_{k+i}$ and $M^{LR}_{ij} = 0$ when $i \neq j$.
\end{itemize}
Figure \ref{matrixfig} illustrates the structure of this matrix $M(\pi_1,\ldots,\pi_n)$.
With the matrix $M(\pi_1,\ldots,\pi_n)$ defined, let
\begin{equation}
\Phi(\pi_1,\ldots,\pi_n) := \det M(\pi_1,\ldots,\pi_n). \label{thepoly}
\end{equation}
(To apply $\Phi$ in the case of \eqref{goal}, one need only specify that $\pi_j = D_x \rho(x,y_j)$ for each $j=1,\ldots,n$.)
Permuting the columns of $M(\pi_1,\ldots,\pi_n)$ brings it exactly into the form identified in Section \ref{blpoly}, so in particular \eqref{thepoly} defines a polynomial $\Phi$ which has the invariance property \eqref{algebra2} and is homogeneous of degree $c$ in each of the matrices $\pi_1,\ldots,\pi_n$ (so $m=n$ and $d_1,\ldots,d_n = c$ in \eqref{algebra1}).
In particular, this quantity \eqref{thepoly} will satisfy \eqref{mainlb} when $\pi_j := D_x \rho (x,y_j)$ for each $j=1,\ldots,n$.

\begin{figure}[ht]
\begin{center}
\begin{tikzpicture}[rotate=90,xscale=-0.16,yscale=-0.16]
\draw [thick] (0,0) rectangle (60,60);	

\draw [dotted] (0,25) -- (35,60);  
\draw [dotted] (0,25) -- (60,25);
\draw [dotted] (35,0) -- (35,60);

\draw (0,25) rectangle (5,32);		\node at (2.5,28.5) {$B_1$};
\draw (5,25) rectangle (10,32);		\node at (7.5,28.5) {$B_2$};
\draw (5,32) rectangle (10,39);		\node at (7.5,35.5) {$B_2$};
\draw (10,32) rectangle (15,39);		\node at (12.5,35.5) {$B_3$};
\draw (10,39) rectangle (15,46);		\node at (12.5,42.5) {$B_3$};
\draw (15,39) rectangle (20,46);		\node at (17.5,42.5) {$B_4$};
\draw (20,39) rectangle (25,46);		\node at (22.5,42.5) {$B_5$};
\draw (20,46) rectangle (25,53);		\node at (22.5,49.5) {$B_5$};
							\node at (27.5,49.5) {$\ddots$};
%\draw (25,46) rectangle (30,53);
%\draw (25,53) rectangle (30,60);
\draw (30,53) rectangle (35,60);		\node at (32.5,56.5) {$B_k$};

\draw (35,25) rectangle (40,32);		\node at (37.5,28.5) {$B_{k+1}$};
\draw (40,32) rectangle (45,39);		\node at (42.5,35.5) {$B_{k+2}$};
\draw (45,39) rectangle (50,46);		\node at (47.5,42.5) {$B_{k+3}$};
							\node at (52.5,49.5) {$\ddots$};
%\draw (50,46) rectangle (55,53);
\draw (55,53) rectangle (60,60);		\node at (57.5,56.5) {$B_{n}$};

\draw (0,0) rectangle (5,5);		\node at (2.5,2.5) {$A_1$};
\draw (5,0) rectangle (10,5);		\node at (7.5,2.5) {$A_2$};
\draw (5,5) rectangle (10,10);		\node at (7.5,7.5) {$A_2$};
\draw (10,5) rectangle (15,10);		\node at (12.5,7.5) {$A_3$};
\draw (10,10) rectangle (15,15);		\node at (12.5,12.5) {$A_3$};
\draw (15,10) rectangle (20,15);		\node at (17.5,12.5) {$A_4$};
\draw (20,10) rectangle (25,15);		\node at (22.5,12.5) {$A_5$};
\draw (20,15) rectangle (25,20);		\node at (22.5, 17.5) {$A_5$};
							\node at (27.5, 17.5) {$\ddots$};
%\draw (25,15) rectangle (30,20);	\node at (27.5, 17.5) {$A_6$};
%\draw (25,20) rectangle (30,25);	\node at (27.5, 22.5) {$A_6$};
\draw (30,20) rectangle (35,25);		\node at (32.5, 22.5) {$A_k$};

\draw (35,0) rectangle (40,5);		\node at (37.5,2.5) {$A_{k+1}$};
\draw (40,5) rectangle (45,10);		\node at (42.5,7.5) {$A_{k+2}$};
\draw (45,10) rectangle (50,15);		\node at (47.5,12.5) {$A_{k+3}$};
							\node at (52.5,17.5) {$\ddots$};
%\draw (50,15) rectangle (55,20);	\node at (52.5,17.5) {$I$};
\draw (55,20) rectangle (60,25);		\node at (57.5, 22.5) {$A_n$};
\end{tikzpicture}
\end{center}
\caption{Illustration of the structure of the matrix $M(\pi_1,\ldots,\pi_n)$. \label{matrixfig}}
\end{figure}
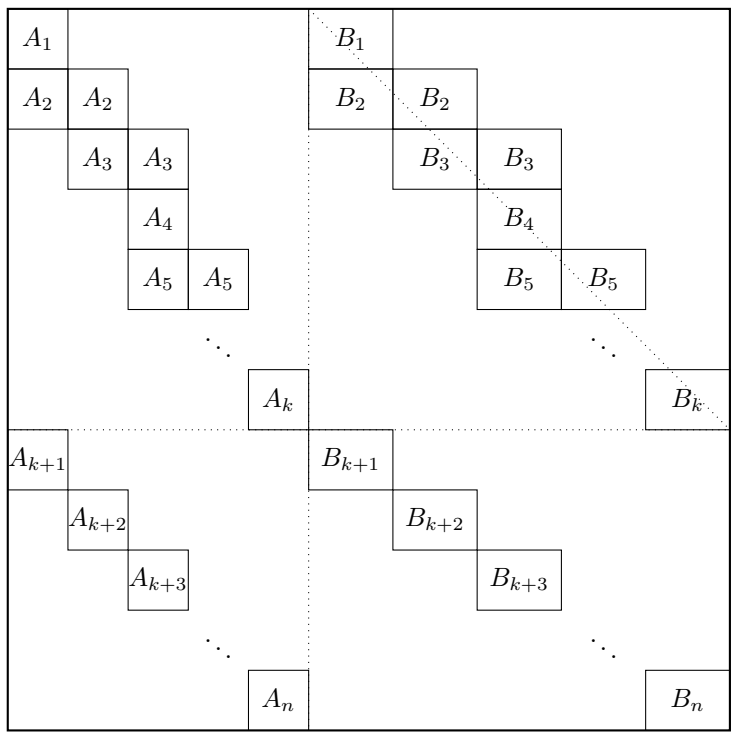

When \eqref{theld} is used for the matrices $\pi_1,\ldots,\pi_n$ as described above, it will be the case that $A_1 = \cdots = A_{n} = I_{c \times c}$ and $B_j = B (x - y_j)$ for each $j=1,\ldots,n$ with
\[ B(t) := \left[ \begin{array}{ccc} t_1 \lambda_{11} & \cdots & t_k \lambda_{1k} \\ \vdots & \ddots & \vdots \\ t_1 \lambda_{c1} & \cdots & t_k \lambda_{ck} \end{array} \right], \]
under the convention that $t = (t_1,\ldots,t_k) \in \R^{k}$ and that each $x-y_j$ is understood to be projected down to $\R^{k}$ by retaining only the first $k$ coordinates of each $x - y_j \in \R^n$.
  Restricting to the situation in which $y_{k+1} = \cdots = y_{n}$, it will be the case that $\pi_{k+1} = \cdots = \pi_{n}$. By elementary row operations and expanding the determinant of $M$, it follows that
\[ \Phi(\pi_1,\ldots,\pi_k,\pi_{n},\ldots,\pi_n) = (-1)^{c^3k} \det M^{UR} (B_1 - B_n,\ldots,B_k - B_n). \]
For convenience, define
\[ \Phi^{UR} (B_1,\ldots,B_k, B_n) := \det M^{UR} (B_1 - B_n,\ldots,B_k - B_n), \]
where by $M^{UR}(B_1 - B_n,\ldots,B_k - B_n)$, we mean simply the matrix with the same structure as $M^{UR}$ but with each $B_1,\ldots,B_k$ replaced by $B_1 - B_n,\ldots,B_k-B_n$, respectively.
On the full diagonal $y_1 = \cdots = y_n$, the matrix $M^{UR}$ will be identically zero, and so $\Phi^{UR}$ will be zero as well.

Since $B(t)$ is some $c \times k$ real matrix which depends smoothly on the parameter $t \in \R^{k}$, one can precisely understand the low-order derivatives of $\Phi^{UR}$ on the diagonal. For each $j=1,\ldots,n$, let $t^{(j)} \in \R^{k}$ denote the first $k$ coordinates of $x - y_j$. The immediate goal is to compute $\Phi^{UR}$ and its low-order derivatives at a point $t^{(1)} = \cdots = t^{(n)} = t^{(0)}$ for some fixed value of $t^{(0)}$. Since $B_{k+1} = \cdots = B_n = B(t^{(0)})$ on the diagonal,
For each index $j \in \{1,\ldots,k\}$, there is a unique collection of $c$ rows of the matrix $M^{UR}(B_1 - B_n,\ldots,B_k-B_n)$ which vanish identically when $t^{(j)} = t^{(0)}$; consequently
\[ \partial^{\alpha_1}_{t^{(1)}} \cdots \partial^{\alpha_k}_{t^{(k)}} \Phi^{UR}(B(t^{(1)}), \ldots,B(t^{(k)}), B(t^{(0)})) = 0 \]
when $t^{(1)} = \cdots = t^{(k)} = t^{(0)}$ if $|\alpha_j| < c$ for any $j = 1,\ldots,k$. This is precisely the situation anticipated by \eqref{vcond}: taking $c_1 = \cdots = c_k = c$ and $c_{k+1} = \cdots = c_n = 0$, then implies that \eqref{simpler1} holds, which then establishes that the inequality \eqref{goal} would in principle be sufficient to prove Theorem \ref{radonthm} by the application of Theorem \ref{genradonthm}.

A precise analysis of higher derivatives of $\Phi^{UR}$ on the diagonal is more delicate.  By linearity of $B$ as a function of $t$, it suffices to assume $t^{(0)} = 0$. To establish a lower bound for quantity $\omega(t)$ from \eqref{betterdens}, one may use \eqref{simpler1} and \eqref{simpler2}. After these reductions, it suffices to compute or otherwise estimate the derivatives of $\Phi^{UR}(B(t^{(1)}),\ldots,B(t^{(k)}),0)$ with respect to constant-coefficient vector fields
 $X_1,\ldots,X_k$ of the form
\[ X_j^{(i)} = \sum_{\ell = 1}^{j} c_{j\ell} \partial_{t^{(i)}_\ell}. \]
These are just the vector fields determined by $U^* \partial$ in \eqref{goal}. In particular, $X_j^{(i)}$ denotes the $j$-th operator among those defining $U^* \partial$, applied to the variable $t^{(i)}$.
Note in particular that $X_1^{(i)}$ points in the first coordinate direction in the variables $t^{(i)}$, $X_2^{(i)}$ lies in the span of the first two coordinate directions, and so on. To simplify computations, it will be assumed for the moment that the diagonal entries $c_{\ell \ell} = U_{\ell \ell}$ are all equal to $1$. It will also be useful to take the periodicity convention $X_{j + Nk}^{(i)} := X_{j}^{(i)}$ for any positive integer $N$.  
When $j$ and $j'$ are both integer subscripts of the vector fields just defined, the relation $j < j'$ will be said to hold when this inequality holds in the usual sense for the representatives of $j,j'$ taken from the interval $\{1,\ldots,k\}$ (i.e., the relation $j < j'$ will mean that the representative of $j$ which belongs to $\{1,\ldots,k\}$ is less than the corresponding representative of $j'$).

 It will be shown by induction on $\ell$ that for any $\ell \leq k$, one has
\begin{equation} \label{mainindid} \begin{split} X^{(\ell)}_{\ell c} & \cdots X^{(\ell)}_{(\ell-1)c+1} \cdots X^{(1)}_{c} \cdots X^{(1)}_{1} \Phi^{UR} (B(t^{(1)}), \ldots,B(t^{(k)}), 0 )   \\ & = \det M^{UR}_{\ell} \prod_{j=0}^{\ell-1} \det \left[ \begin{array}{ccc} \lambda_{1(jc+1)} & \cdots & \lambda_{1 (j c + c)} \\ \vdots & \ddots & \vdots \\ \lambda_{c(j c + 1)} & \cdots & \lambda_{c(jc + c)} \end{array} \right], 
\end{split}  \end{equation}
where $M^{UR}_{\ell}$ is the $(k - \ell) c \times (k-\ell)c$ lower-right minor of the matrix $M^{UR}$ and where the columns of the matrix $\lambda$ of coefficients associated to the operator \eqref{myoperator} are regarded as periodic with period $k$ just as was the case for the index $j$ of the vectors $X^{(i)}_j$.  There are two cases to consider: one case when the block $B_{\ell+1}$ appears exactly once in the matrix $M^{UR}$ (e.g., $B_1$ or $B_4$ in Figure \ref{matrixfig}) and another case when the block appears twice in $M^{UR}$ with one copy appearing immediately to the right of the other (e.g., $B_2$ or $B_3$ in Figure \ref{matrixfig}).  In the first case, the truncated matrix $M^{UR}_\ell$ has the $c \times 1$ block
\[ \left[ \begin{array}{ccc} t_{\ell c + 1}^{(\ell+1)} \lambda_{1 (\ell c +1)} & \cdots & t_{\ell c + 1} ^{(\ell+1)} \lambda_{ c (\ell c + 1)} \end{array} \right]^T \]
in its upper left-hand corner. As a function of $t^{(\ell+1)}$, the determinant $\det M^{UR}_\ell$ does not depend on $t^{(\ell+1)}_i$ for any $i < \ell c + 1$ (interpreted periodically), since all such columns of $M^{UR}$ that do depend on these variables lie outside the minor $M^{UR}_\ell$. This means that the derivative of $\det M^{UR}_\ell$ with respect to $X^{(\ell+1)}_{\ell c + 1}$ must simply equal the derivative with respect to ${t^{(\ell)}_{\ell c + 1}}$, the effect of which is to replace the upper left block in the first column with the new block
\[ \left[ \begin{array}{ccc}  \lambda_{1(\ell c + 1)} & \cdots & \lambda_{ c(\ell c + 1)} \end{array} \right]^T \]
and to replace all other entries in the first column with zeros (if they do not vanish already) because they are constant with respect to $t^{(\ell+1)}$. The argument then repeats for all the remaining derivatives $X_{\ell c + 1}^{(\ell+1)}$ through $X_{(\ell+1)c}^{(\ell+1)}$ by advancing to the second column and so on. At each stage, there is no dependence on $t^{(\ell)}$ with respect to any ``lower'' coordinate directions.  Once all derivatives of $\det M^{UR}_{\ell}$ with respect to $t^{(\ell+1)}$ have been taken, the result is that
\[ X_{\ell c + c}^{(\ell+1)} \cdots X_{\ell c + 1}^{(\ell+1)} \det M^{UR}_{\ell} \]
may be expressed as
the determinant of a matrix with a $c \times c$ minor 
in the upper-left corner equalling
\[ \left[ \begin{array}{ccc} \lambda_{1(\ell c+1)} & \cdots & \lambda_{1 (\ell c + c)} \\ \vdots & \ddots & \vdots \\ \lambda_{c(\ell c + 1)} & \cdots & \lambda_{c( \ell c + c)} \end{array} \right] \]
 and the matrix $M^{UR}_{\ell+1}$ in the lower lower right corner.

On the other hand, if the block $B_{\ell+1}$ appears twice in $M^{UR}$, then the argument above requires slight modification.  First, there must be an index $p$ in the range $\{\ell c + 1, \ldots, (\ell+1) c\}$ which is equivalent to $1$ modulo periodicity.  If any columns of the leftmost $B_{\ell+1}$ block appear in the minor $M^{UR}_{\ell}$, they must appear alone on their own column since no block in $M^{UR}$ can have neighbors both on the right and below.  This would mean that $M^{UR}_{\ell}$ has a block in the upper left hand corner with the form
\[ \left[ \begin{array}{ccc} t_{\ell c + 1}^{(\ell+1)} \lambda_{1 (\ell c + 1) } & \cdots &  t_{p-1}^{(\ell+1)} \lambda_{1 (p-1)} \\ \vdots & \ddots & \vdots \\
t_{\ell c + 1} ^{(\ell+1)} \lambda_{ c(\ell c + 1)}  & \cdots & t_{p-1} ^{(\ell+1)} \lambda_{ c (p-1) }  \end{array} \right] \]
and all other entries in these same columns must be zero.
It follows when taking the determinant of $M^{UR}_{\ell}$ that factors of $t^{(\ell+1)}_{\ell c + 1}, \ldots, t^{(\ell+1)}_{p-1}$ appearing on their own rows simply factor out by multilinearity of the determinant as a function of the columns.  Furthermore, although these same columns of the leftmost $B_{\ell+1}$ appear again in the rightmost $B_{\ell+1}$ block, elementary column operations allow one to subtract the leftmost copy of these columns from the rightmost block without changing the determinant of $M^{UR}_{\ell}$. Thus it may be assumed without loss of generality that $\det M^{UR}_{\ell}$ has no dependence on $t^{(\ell+1)}_{ \ell c + 1}, \ldots, t^{(\ell+1)}_{p-1}$ beyond the factors already obtained from the initial columns. By exactly the same argument as above, then, it follows that
\[ \begin{aligned} X^{(\ell+1)}_{(\ell+1)c} &  \cdots X^{(\ell+1)}_{p} \det M^{UR}_{\ell} \\ & = t^{(\ell-1)}_{\ell c + 1} \cdots t^{(\ell+1)}_{p-1} \det M^{UR}_{\ell+1} \det \left[ \begin{array}{ccc} \lambda_{1(\ell c + 1)} & \cdots & \lambda_{1 (\ell c + c) } \\ \vdots & \ddots & \vdots \\ \lambda_{c (\ell c + 1)} & \cdots & \lambda_{c(\ell c + c)} \end{array} \right] \end{aligned},  \]
and from this identity the desired conclusion holds after differentiating once again with respect to the remaining derivatives $X^{(\ell+1)}_{\ell c + 1}, \ldots, X^{(\ell+1)}_{p-1}$ in order just listed ($X_{(\ell c + 1)}^{(\ell+1)}$ first, etc.), once again using the fact that at every step, there is no dependence on variables from the ``lower'' coordinate directions. Finally, because the $X$ vector fields are constant-coefficient linear combinations of coordinate vector fields, we see that while the order of differentiation was extremely useful to exploit for computational purposes, it does not have an effect on the final result. Therefore in both cases we conclude that
\[ X^{(\ell+1)}_{(\ell+1)c} \cdots X^{(\ell+1)}_{\ell c + 1} \det M^{UR}_{\ell} = \det M^{UR}_{\ell+1} \det  \left[ \begin{array}{ccc} \lambda_{1(\ell c+1)} & \cdots & \lambda_{1 (\ell c + c)} \\ \vdots & \ddots & \vdots \\ \lambda_{c(\ell c + 1)} & \cdots & \lambda_{c( \ell c + c)} \end{array} \right].
 \]
Now \eqref{mainindid} with $\ell = k$ gives the final conclusion that
\begin{equation} \begin{aligned} X^{(1)}_1 & \cdots X^{(1)}_c \cdots X^{(k)}_{(k-1) c + 1} \cdots X^{(k)}_{k c} \Phi^{UR} (B(t^{(1)}), \ldots,B(t^{(k)}), 0 )   \\ & =  \prod_{j=0}^{k-1} \det \left[ \begin{array}{ccc} \lambda_{1 (jc+1)} & \cdots & \lambda_{1(j c + c) } \\ \vdots & \ddots & \vdots \\ \lambda_{c(j c + 1) } & \cdots & \lambda_{c(jc + c)} \end{array} \right]. 
\end{aligned} \label{themainid} \end{equation}

The inequality \eqref{themainid} gives exactly the desired inequality \eqref{goal}, i.e., 
\begin{align*} \max_{|\alpha_1| = \cdots = |\alpha_k| = c} & | (U^* \partial)^{\alpha_1}_1 \cdots (U^*)^{\alpha_k}_k \Phi (D_x \rho(x,y) ,\ldots,D_x \rho(x,y))| \\ & \geq \left| \prod_{j=0}^{k-1} \det \left[ \begin{array}{ccc} \lambda_{1 (jc+1)} & \cdots & \lambda_{1(j c + c) } \\ \vdots & \ddots & \vdots \\ \lambda_{c(j c + 1) } & \cdots & \lambda_{c(jc + c)} \end{array} \right] \right|  \end{align*}
under the assumption that $U_{ii} = 1$ for each $i$. When the diagonal elements of $U$ are \textit{not} all $1$; one may instead apply \eqref{themainid} by choosing
\[ X_j^{(i)} = \sum_{\ell=1}^j U_{jj}^{-1} U_{j \ell} \partial_{t_\ell^{(i)}} \]
for each $i=1,\ldots,k$ and $j=1,\ldots,k$.
Because each subscript index in the set $\{1,\ldots,k\}$ appears exactly $c$ times among the derivatives on the left-hand side of \eqref{themainid}, multiplying both sides of \eqref{themainid} by $|\det U|^c$ (which is simply the $c$-fold product of the absolute value of the diagonal elements of $U$) gives the more general inequality
\begin{align*} \max_{|\alpha_1| = \cdots = |\alpha_k| = c} & | (U^* \partial)^{\alpha_1}_1 \cdots (U^*)^{\alpha_k}_k \Phi (D_x \rho(x,y) ,\ldots,D_x \rho(x,y))| \\ & \geq |\det U|^c \left| \prod_{j=0}^{k-1} \det \left[ \begin{array}{ccc} \lambda_{1 (jc+1)} & \cdots & \lambda_{1(j c + c) } \\ \vdots & \ddots & \vdots \\ \lambda_{c(j c + 1) } & \cdots & \lambda_{c(jc + c)} \end{array} \right] \right|  \end{align*}
for arbitrary invertible upper-triangular matrix $U$; if $U$ is not invertible, the inequality just established is trivially true.  This is exactly the desired inequality \eqref{goal}.

By \eqref{simpler1} and \eqref{simpler2} (fixing $s = c$), it follows that the appropriate density $\omega(t)$ from \eqref{simpler1} is at least bounded below by a fixed implicit constant (depending only on $n$) times $K^{1/c}$, where
\[ K := \left| \prod_{j=0}^{k-1} \det \left[ \begin{array}{ccc} \lambda_{1 (jc+1)} & \cdots & \lambda_{1(j c + c) } \\ \vdots & \ddots & \vdots \\ \lambda_{c(j c + 1) } & \cdots & \lambda_{c(jc + c)} \end{array} \right] \right|. \]
By Theorem \ref{betterthm}, the measure $K^{1/c} dt$ satisfies $K^{1/c} dt \leq \omega dt$, so that
\[ \sup_{y_1,\ldots,y_m \in F \cap \li{x}} | \Phi(\{D_x \rho(x,y_j)\})| \gtrsim \left[ K^{1/c} \sigma(F \cap \li{x}) \right]^c \]
for all Borel sets $F \subset \R^{n}$. Assuming that $K > 0$, the inequality \eqref{theradonineq} must hold by Theorem \ref{genradonthm} after fixing $m = n$ and $s = n-k$. This is exactly the desired conclusion of Theorem \ref{radonthm}.

\subsection{A Generalization}

The nature of nonconcentration inequalities such as the main hypothesis \eqref{nonconcentrate} of Theorem \ref{genradonthm} is that when \eqref{nonconcentrate} can be shown to for some model operator, this can often be used to show that it must hold for some generic class of operators and that there must exist some nontrivial polynomial functions of the data which govern the sort of nondegeneracy which \eqref{nonconcentrate} implicitly requires. The following result gives such an example:
\begin{theorem}
Let $k$ and $n$ be positive integers satisfying the inequalities $k < n \leq 2k$ and let all vectors $x,y \in \R^n$ be regarded as pairs $(x',x'') \in \R^{k} \times \R^{n-k}$ and $(y',y'') \in \R^{k} \times \R^{n-k}$, respectively. There exists a nonempty collection of nontrivial polynomials $\{P_1,\ldots,P_N\}$ on the space $(\R^{k \times k})^{n-k}$ (i.e., on the space of $(n-k)$-tuples of $k \times k$ real matrices) such that the following holds: For any incidence relation $\rho$ of the form \label{quantthm}
\[ 
\rho(x,y) :=  y'' - x'' - Q(x',y')\]
where $Q : \R^{k} \times \R^{k} \rightarrow \R^{n-k}$ is a polynomial in $x'$ and $y'$, if $\Omega' \subset \R^k \times \R^k$ is an open set such that
\[ \sum_{i=1}^N |P_i ( \partial^2_{x'y'} Q(x',y')))|^2 \geq c \]
at every point $(x',y') \in \Omega'$ for some constant $c > 0$, then for any Borel set $E \subset \R^n$, the Radon-like operator
\[ T f(x) := \int_{(x',y') \in \Omega'} f(y', x'' + Q(x',y')) dy' \]
satisfies
\[ || T \chi_E ||_{L^{\frac{2n-k}{n-k}}(\R^n)} \leq C |E|^{\frac{n}{2n-k}} \]
for some $C < \infty$ independent of $E$ (where $|E|$ denotes Lebesgue measure of $E$).
\end{theorem}
\begin{proof}
As noted above, let $x = (x',x'') \in \R^{k} \times \R^{n-k}$ and similarly for $y$. 
Consider the Radon-like operator parametrized by $y = x + (t,Q(x',x'+t))$ for $t \in \R^{k}$, which has defining function $\rho(x,y) := y'' - x'' - Q(x',y')$ as noted in the statement of the theorem. Using Theorem \ref{betterthm} and 
following the same initial derivation as in the proof of Theorem \ref{radonthm}, to verify the main hypothesis of Theorem \ref{genradonthm}, it suffices to show that
\[ \begin{aligned} \Phi & (\{ D_x  \rho(x,y_j)\}_{j=1}^n) \\ & = \det M^{UR}( D_{x'} Q(x',y'_1) - D_{x'} Q(x',y'),\ldots, D_{x'} Q(x',x'+y'_k - D_{x'} Q(x',y')) \end{aligned} \]
has the property that
\begin{equation} \max_{T \in \GL{n-k}} \max_{|\alpha_1| = \cdots = |\alpha_k| = n-k} \frac{| (T^* \partial')_1^{\alpha_1} \cdots (T^* \partial')_k^{\alpha_k} \Phi  (\{ D_x  \rho(x,y)\}_{j=1}^n)|^{\frac{1}{n-k}}}{|\det T|} \label{theq} \end{equation}
is uniformly bounded below for all $(x',y') \in \Omega' \subset \R^{k \times k}$, where $\partial'$ represents the partial derivatives with respect to the single-primed $y'$-variables. As before, note once again that $|\alpha_j| < n-k$ for some $j \in \{1,\ldots,k\}$, the matrix $M^{UR}$ will have a row which is identically zero when it is evaluated on the diagonal $y_1 = \cdots = y_n = y$; when $|\alpha_j| = n-k$ for each $j \in \{1,\ldots,n-k\}$, the resulting derivative $(T^* \partial')_1^{\alpha_1} \cdots (T^* \partial')_k^{\alpha_k} \Phi  (\{ D_x  \rho(x,y)\}_{j=1}^n)$ is expressible on the diagonal as a polynomial function of $\partial^2_{x'y'} Q$ simply because each derivative must fall on a distinct row of $M^{UR}$ for the determinant to be nonzero, which means that no higher-order derivatives in $y'$ occur in nonzero terms.  If $R$ is any polynomial function of the quantities
\[ \{ {\partial_1'}^{\alpha_1} \cdots {\partial_k'}^{\alpha_k} \Phi(\{D_x \rho(x,y)\}) \}_{|\alpha_1| = \cdots = |\alpha_k| = m}  \]
which is invariant under the natural action of $T \in \SL{n-k}$, then just as in the proof of Theorem \ref{radonthm}, it must be the case that
\begin{align*} \max_{T \in \GL{n-k}}&  \max_{|\alpha_1| = \cdots = |\alpha_k| = n-k} \frac{| (T^* \partial')_1^{\alpha_1} \cdots (T^* \partial')_k^{\alpha_k} \Phi  (\{ D_x  \rho(x,y)\}_{j=1}^n)|^{\frac{1}{n-k}}}{|\det T|} \\ & \gtrsim |R(  \{ {\partial_1'}^{\alpha_1} \cdots {\partial_k'}^{\alpha_k} \Phi(\{D_x \rho(x,y)\}) \}_{|\alpha_1| = \cdots = |\alpha_k| = m} ) |^{\frac{1}{n-k}} 
\end{align*}
for some implicit constant that depends only on $n$, $k$, and $R$.  Because we know that the quantity \eqref{theq} on the left-hand side is nonzero for \textit{some} choice of $\rho$ (namely, the case established by Theorem \ref{radonthm}), this guarantees that it is possible to find a nontrivial invariant polynomial $R$ because the null cone of the $\SL{n-k}$ representation associated to \eqref{theq} does not trivially contain all vectors. Taking $P ( \partial^2_{x'y'} Q(x',y')))$ to equal $R(  \{ {\partial_1'}^{\alpha_1} \cdots {\partial_k'}^{\alpha_k} \Phi(\{D_x \rho(x,y)\}) \}_{|\alpha_1| = \cdots = |\alpha_k| = m} ) $ for all possible nontrivial $R$ establishes the conclusion of this theorem.
\end{proof}

\subsection{Maximal codimension}
\label{radonmaxsec}

The final application of Theorem \ref{genradonthm} is to establish boundedness of certain non-translation-invariant quadratic model operators which have the maximum possible codimension for the given dimension. When the dimension of the underlying submanifold is $k$, the codimension cannot exceed $k^2$, which is simply equal to the number of mixed partial derivatives $\partial^2_{x'y'}$.  

Let $x := (x',x'')$ for $x' \in \R^{k}$ and $x'' \in \R^{k^2}$.  For convenience, $x_i'$ will denote the coordinates of $x'$ in the standard basis and $x''_{ij}$ will be the coordinates of $x''$, where $i,j$ range over $\{1,\ldots,k\}$. The operator which will be studied here is given by the definition
\begin{equation} T f(x) := \int_{\R^k} f( x' + t, \{x_{ij} + x_i' (t_j + x_j) \}_{i,j =1}^n ) dt \label{thecodimop} \end{equation}
for all measurable functions on $\R^{k} \times \R^{k^2}$. The associated defining function $\rho(x,y)$ maps into $\R^{k^2}$ and has
\[\rho(x,y) :=  -y'' + x'' + \{ x_i' y_j' \}_{i,j=1}^{k}. \]
\begin{theorem}
The Radon-like operator given by \eqref{thecodimop} is of restricted strong type $(\frac{2k+1}{k+1}, \frac{2k+1}{k})$.
\end{theorem}
\begin{proof}
The matrix $D_x \rho(x,y)$ consists of two blocks: one $k^2 \times k$ block on the left and a $k^2 \times k^2$ block on the right which simply equals the $k^2 \times k^2$ identity matrix. The block on the left can itself be understood as composed of $k \times 1$ sub-blocks which equal $y'$ (interpreted as a column matrix) along the block diagonal and $0$ elsewhere, i.e., in row $(i,j)$ and column $\ell$, the entry of this matrix is $y'_{j} \delta_{i,\ell}$ with $\delta$ being the Kronecker $\delta$. The simplest invariant polynomial which may be used to estimate the Brascamp-Lieb weight is the following:
\begin{align*} \Phi(\{D_x & \rho(x,y_j)\}_{j=1}^{k+1}) \\ &  := \det \left[ \begin{array}{cccc} D_x \rho(x,y_1) & 0 & \cdots & 0 \\ 0 & D_x \rho(x,y_2) & \ddots & \vdots \\
\vdots & \ddots & \ddots & 0 \\ 0 & \cdots & 0 & D_x \rho(x,y_k) \\ D_x \rho(x,y_{k+1}) & \cdots & \cdots & D_x \rho(x,y_{k+1}) \end{array} \right]. 
\end{align*}
To compute this determinant, subtract one copy of each of the upper block rows from the bottom block row and expand the determinant in those columns corresponding to the $k^2 \times k^2$ identity blocks of $D_x \rho(x,y_1),\ldots,D_x \rho(x,y_k)$; since there are now no nonzero entries in these columns in the final block row, the expansion is trivial and one concludes that, up to a possible factor of $\pm 1$, the determinant equals
\[ \det \left[ \begin{array}{cccccccc}
y_{k+1} - y_1 & 0 & \cdots & 0 & y_{k+1} - y_k & 0 & \cdots & 0 \\
0 & \ddots & \ddots & \ddots & \ddots  & \ddots & \ddots & \vdots \\
\vdots & \ddots & \ddots & \ddots & \ddots & \ddots & \ddots & 0 \\
0 & \cdots & 0 & y_{k+1} - y_1 & 0 & \cdots & 0 & y_{k+1} - y_k
\end{array} \right] \]
where each $y_{k+1} - y_j$ is understood as a $k \times 1$ block, as is each $0$. Rearranging columns, this matrix can itself be brought into block form, and consequently
\[ | \Phi(\{D_x \rho(x,y_j)\}_{j=1}^{k+1})| = \left| \det \left[ \begin{array}{ccc} y_{k+1} - y_1 & \cdots & y_{k+1} - y_k \end{array} \right] \right|^k. \]
Up to the factor of $k$, this $\Phi$ corresponds to the case of a multilinear determinant functional, which has been studied in a variety of contexts \cite{gressman2010}.  In particular, it is known (see \cite{gressman2018}) that
\[ \sup_{y_1,\ldots,y_{k+1} \in F} |\det \left[ \begin{array}{ccc} y_{k+1} - y_1 & \cdots & y_{k+1} - y_k \end{array} \right]| \gtrsim |F| \]
for any Borel set $F \in \R^k$, so it follows that
\[ \sup_{y_1,\ldots,y_{k+1} \in F \cap \li{x} } | \Phi(\{D_x \rho(x,y_j)\})| \gtrsim |\sigma(\li{x} \cap F)|^{k}. \]
By \eqref{blblock2} with $m = k+1$, $n = k(k+1)$ and $d_1 = \cdots = d_{k+1} = k^2$ (one can see that the exponent is $k^2$ by using multilinearity of the determinant defining $\Phi$ as a function of its rows), it follows that that $|\blw(\{D_x \rho(x,y_j)\}_{j=1}^{k+1}| \gtrsim |\Phi(\{D_x \rho(x,y_j)\}_{j=1}^{k+1})|$, so Theorem \ref{genradonthm} applies when $s = k$ to give that
\[ || T \chi_E ||_{L^{\frac{2k+1}{k}}(\R^{k} \times \R^{k^2})} \lesssim |E|^{\frac{k+1}{2k+1}} \]
for all Borel $E \subset \R^{k} \times \R^{k^2}$.
\end{proof}

\section{Appendix}
\label{appendix}

This Appendix contains the proof of Lemma \ref{specialdf}, which establishes the existence of a ``normalized'' defining function which satisfies a number of desirable properties. Lemma \ref{specialdf} was used in Section \ref{lastksec} to complete the proof of Theorem \ref{kakeyathm}. The proof of Lemma \ref{specialdf} is essentially a consequence of a quantitative version of the Implicit Function Theorem.

To simplify matters somewhat, it is useful to adopt some additional notation. For any $x \in \R^n$ and any $r > 0$, let $Q_{x,r} := x + (-r,r)^n$.vFix $| \cdot |$ to be the $\ell^\infty$ norm $\R^n$ in the standard coordinates and further fix $|| \cdot ||$ to be the $\ell^\infty \rightarrow \ell^\infty$ operator norm on matrices in $\R^{n \times n}$. There is no intrinsic reason why such a choice is required, but having norm balls equal to product boxes makes the application of these results somewhat simpler.

\begin{proposition}
Let $\Phi$ be an everywhere differentiable map from the ball $Q_{x_0,r}$ into $\R^{n-k}$, where $0 \leq k < n$. Let $D \Phi_x$ be the $(n-k) \times n$ derivative matrix of $\Phi$ at $x$ and let $R$ be an $n \times (n-k)$ matrix such that \label{ift}
\[ \sup_{x \in Q_{x_0,r}}  ||  D \Phi_x R  - I || \leq  c < 1. \]
If $|\Phi(x_0)| < r ||R||^{-1} (1-c)$, there exists some $u \in \R^n$ such that the point $x = x_0 + R u$ satisfies $x \in Q_{x_0,r}$, $\Phi(x) = 0$, and $| x - x_0 | \leq ||R|| (1-c)^{-1} | \Phi(x_0)|$. 
\end{proposition}
\begin{proof}
The point $x$ will be the limit of the sequence given by 
\[ x_{j+1} := x_j - R   \Phi(x_j) \]
for all $j \geq 0$.  By assumption,  $| \Phi(x_0) | < r ||R||^{-1} (1-c)$. Suppose that for some value of the index $j$, it is known that the following inequalities hold:
\[
\begin{aligned}
| \Phi(x_j)| & \leq c^j | \Phi(x_0)|, \\
|x_j - x_{0}| & \leq ||R|| \frac{1 -c^j}{1-c} |\Phi(x_0)| < r (1-c^j).
\end{aligned}
\] 
By definition of $x_{j+1}$ and the above inequality for $|\Phi(x_j)|$, 
\begin{equation} |x_{j+1} - x_j| \leq ||R|| c^{j} |\Phi(x_0)|  \label{makescauchy} \end{equation}
which gives that
\[ |x_{j+1} - x_0| \leq |x_{j} - x_0| + ||R|| c^{j} |\Phi(x_0)| \leq ||R|| \frac{1-c^{j+1}}{1-c} | \Phi(x_0) | < r (1 - c^{j+1}).\]
One consequence of this inequality is that the line segment joining $x_j$ and $x_{j+1}$ belongs to $Q_{x_0,r}$. Consequently, the function 
\[ t \mapsto \Phi( x_j - t R \Phi(x_j) ) \]
is well-defined and differentiable for all $t$ in some open interval containing $[0,1]$. By the chain rule and the Mean Value Theorem, for any $z \in \R^n$, there is some $t \in [0,1]$ such that 
\[ \ang{z,  - D \Phi_{x_j - t R \Phi(x_j)} R \Phi(x_j) } = \ang{z,    \Phi(x_{j+1}) - \Phi(x_j)} , \]
where $\ang{ \cdot, \cdot}$ is the usual inner product in standard coordinates.
For convenience, let $x_* := x_j - t R  \Phi(x_j)$. Rearranging terms in the above expression yields
\[  \ang{ z ,  \Phi(x_{j+1}) } = -\ang{ z  , (D \Phi_{x_*} R - I ) \Phi(x_j) }. \]
Taking absolute values and a supremum over all $z$ with coordinates whose magnitudes sum to $1$ and applying the main hypothesis of this proposition gives that $| \Phi(x_{j+1}) | \leq c | \Phi(x_j)|$, which implies that the induction hypotheses continue to hold when the index $j$ is replaced by $j+1$. By \eqref{makescauchy}, the sequence $\{x_j\}$ must be Cauchy; by continuity of $\Phi$, defining $x := \lim_{j \rightarrow \infty} x_j$ gives that $\Phi(x) = \lim_{j \rightarrow \infty} \Phi(x_j) = 0$. The definition of the sequence and continuity of matrix multiplication gives that $x - x_0 = R u $ for some $u \in \R^n$, and the limit of the induction hypotheses gives that  $|x - x_0| \leq ||R|| (1-c)^{-1} | \Phi(x_0)| < r$. 
\end{proof}

\begin{proposition}
Let $\Phi$, $R$, $x_0$, and $r$ be as in Proposition \ref{ift} and suppose $k > 0$.  Let $V$ be the orthogonal complement of the image space of $R$ and suppose  \label{ift2}
\[ \mathop{\sup_{x \in Q_{x_0,r}}}_{|v| \leq 1, \ v \in V}  | D \Phi_x v | \leq C. \]
If $| \Phi(x_0)| < \frac{r}{3} ||R||^{-1} (1-c)$, then 
\[ \mathcal H^{k} ( \set{ x \in Q_{x_0,r} }{ \Phi(x) = 0} ) \geq c_{n} r^k \left( \min\left\{ \frac{1}{2}, \frac{1-c}{6 C ||R||} \right\} \right)^k \]
for some constant $c_n > 0$ that depends only on $n$.
\end{proposition}
\begin{proof}
The new hypothesis guarantees that $| \Phi(x_0 + v)  - \Phi(x_0)| \leq C |v|$ whenever $v \in V$ and $|v| < r$.
The proof is by the Mean Value Theorem as it just appeared: 
\[ \left| \ang{z,  \Phi(x_0 + v) - \Phi(x_0) }  \right| = \left| \ang{z, D \Phi_{x_*} v} \right| \]
for some $x_* \in Q_{x_0,r}$; applying the new hypothesis of this proposition and taking a supremum over $z$ gives $| \Phi(x_0  + v) - \Phi(x_0)| \leq C |v|$. 

Suppose now that $|\Phi(x_0)| < \frac{r}{3} ||R||^{-1} (1-c)$ as assumed in the statement of this proposition. For any $v \in V$ such that $|v| \leq \min\{\frac{r}{2}, \frac{r}{6 C} ||R||^{-1} (1-c) \}$,
\[ |\Phi(x_0 + v) - \Phi(x_0)| \leq C |v| \leq \frac{r}{6} ||R||^{-1} (1-c), \]
which means that $|\Phi(x_0 + v)| \leq |\Phi(x_0)| + \frac{r}{6} ||R||^{-1} (1-c) < \frac{r}{2} ||R||^{-1} (1-c)$.  Moreover, $Q_{x_0 + v, r/2} \subset Q_{x_0,r}$, so the previous proposition applies on the box with new center $x_0 + v$ and new radius $r/2$. This implies that there exists $u \in \R^n$ such that $\Phi(x_0 + v + R u) = 0$ and $|R u| \leq \frac{r}{2}$.
In other words, the zero set $\set{ x \in Q_{x_0,r}}{ \Phi(x) = 0}$ must contain a graph over the  $k$-dimensional set $\set{v \in V}{ |v| \leq \min\{\frac{r}{2}, \frac{r}{6 C} ||R||^{-1} (1-c) \} }$, which forces the graph to have $k$-dimensional Hausdorff measure at least as large as the $k$-dimensional Hausdorff measure of the parametrizing set. This establishes the conclusion of the proposition.
\end{proof}

\begin{lemma}
Suppose $\rho$ is a smooth defining function on some open set $\Omega$ of an incidence relation $\inc$. There exists some open set $\tilde \Omega \subset \Omega$ containing $\inc$ and another smooth defining function $\tilde \rho$ of $\inc$ such that the following hold: \label{specialdf}
\begin{enumerate}
\item At every point $(x,y) \in \inc$, the matrix $D_x \tilde \rho(x,y)$ has rows which are orthonormal vectors in $\R^n$.
\item At every point $(x,y) \in \inc$,
\[ 1 = \det D_x \tilde \rho (D_x \tilde \rho)^T \mbox{ and } \frac{ \det D_y \rho (D_y \rho)^T}{\det D_x \rho (D_x \rho)^T} = \det D_y \tilde \rho (D_y \tilde \rho)^T. \]
\item For every compact subset $K \subset \inc$, there is an open set $U \subset \tilde \Omega$ containing $K$ and a positive $\delta_0$ such that for any $(x,y) \in U$, $|\tilde \rho(x,y)| < \delta \kappa_n$ for any $\delta \leq \delta_0$ (where $\kappa_n$ is some fixed constant depending only on $n$) implies that
\[ \mathcal H^k ( Q_{x,\delta} \cap \ri{y} ) \geq c_n \delta^k \]
for some positive $c_n$ depending only on $n$.
\end{enumerate}
\end{lemma}
\begin{proof}
For any real symmetric positive-definite matrix $A$, let $A^{-1/2}$ be the matrix such that every eigenvector $e$ of $A$ with eigenvector $\lambda > 0$ of $A$ is also an eigenvector with eigenvalue $\lambda^{-1/2}$ of $A^{-1/2}$.
It is relatively easy to see that the mapping $A \mapsto A^{-1/2}$ is a smooth function of $A$; the standard way to see this is to use the identity
\[ A^{-1/2} = \frac{1}{2 \pi i} \int_{\gamma} z^{1/2} (z I  - A)^{-1} dz \]
where $z^{1/2}$ is a branch of the square root on the right half space ${ \mathop{Re} z > 0}$ which equals the positive square root on the real axis and $\gamma$ is, for example, a closed circular contour in the right half space which encloses all eigenvalues of $A$.

Let $\tilde \Omega \subset \Omega$ be the neighborhood of $\inc$ on which $\det D_x \rho (D_x \rho)^T > 0$; the function
\[ \tilde \rho (x,y) := (D_x \rho (D_x \rho)^T)^{-1/2} \rho(x,y) \]
is well-defined and smooth on $\tilde \Omega$ provided that $\rho$ is smooth. This mapping $\tilde \rho$ vanishes if and only if $\rho$ vanishes (so that $\inc$ is also the set of points $(x,y)$ where $\tilde \rho(x,y) = 0$), and by the product rule, $D_x \tilde \rho = (D_x \rho (D_x \rho)^T)^{-1/2} D_x \rho$ at all points of $\inc$ (since all terms in which derivatives fall on $(D_x \rho (D_x \rho)^T)^{-1/2}$ vanish because $\rho$ vanishes). This implies that $D_x \tilde \rho (D_x \tilde \rho)^T$ is the identity matrix at all points $(x,y) \in \inc$, which means that the rows of $D_x \tilde \rho$ are mutually orthogonal unit vectors when $(x,y) \in \inc$.  The formula for $\det D_y \tilde \rho (D_y \tilde \rho)^T$ also follows directly from the definition of $\tilde \rho$.

Now fix any compact subset $K \subset \inc$.  Because $K$ is compact, there must exist some $r > 0$ such that $Q_{x_0,3r} \times Q_{y_0,3r} \subset \tilde \Omega$ for any $(x_0,y_0) \in K$. 
It may further be assumed (after possibly reducing the value of $r$) that
\[ || D_x \tilde \rho(x',y') (D_x \tilde \rho(x_0,y_0))^T - I || < \frac{1}{2} \]
and 
\[ | D_x \tilde \rho(x',y') v | \leq \frac{1}{2} \mbox{ for all } v \in \ker D_x \tilde \rho(x_0,y_0) \mbox{ such that } |v| \leq 1 \]
whenever $(x_0,y_0) \in K$ and $(x',y') \in \tilde \Omega$ are any points that satisfy $|x_0 - x'| < 2r$ and $|y_0 - y'| < 2r$
(simply because the quantities on the left-hand sides of these inequalities will be identically zero when $(x_0,y_0) = (x',y')$  and are continuous functions on compact sets, so are consequently uniformly continuous). 

Now suppose $U$ is the open set of pairs $(x,y)$ such that $|x-x_0| < r$ and $|y-y_0| < r$ for some $(x_0,y_0) \in K$. For any $(x,y) \in U$, fixing $R := (D_x \tilde \rho(x_0,y_0))^T$ gives that %$||R||^{-1} \leq 1$ and that
\[ \sup_{x' \in Q_{x,\delta}} || D_x \rho(x',y) R - I || \leq \frac{1}{2} \mbox{ and } \mathop{\sup_{x' \in Q_{x,\delta}}}_{|v| \leq 1, v \in \ker R^T} | D_x \rho(x',y) v| \leq \frac{1}{2} \]
for any $\delta < r$. Because $R$ consists of orthonormal columns, there must be a constant $\kappa_n' > 0$ depending only on $n$ such that $||R||^{-1} \geq \kappa_n'$.
By Propositions \ref{ift} and \ref{ift2} (taking $c = C = \frac{1}{2}$)
It follows that
\[ |\tilde \rho(x,y)| < \frac{\delta \kappa_n'}{6} \Rightarrow \mathcal H^k( Q_{x,\delta} \cap \ri{y}) \geq c_n \delta^k. \]
The lemma is complete by simply fixing $\delta_0 := r$ and $\kappa_n := \kappa'_n/6$.
\end{proof}

\bibliography{mybib}

\end{document}